\documentclass[reqno, 12pt]{amsart}

\usepackage[utf8]{inputenc}
\usepackage[T1]{fontenc}
\usepackage{lmodern}
\usepackage{amsmath}
\usepackage{amsthm}
\usepackage{amssymb}
\usepackage{enumerate}
\usepackage{mathrsfs}
\usepackage{mathscinet}
\usepackage{stmaryrd} 
\usepackage{graphicx}
\usepackage[all,cmtip]{xy}
\usepackage{url}
\usepackage{hyperref}
\usepackage[a4paper, hmargin=3cm, vmargin=3cm]{geometry}

\newcommand\mathens[1]{\mathbb{#1}} 

\newcommand{\ud}{\mathrm{d}}

\newcommand{\N}{\mathens{N}}
\newcommand{\Z}{\mathens{Z}}

\newcommand{\R}{\mathens{R}}
\newcommand{\C}{\mathens{C}}

\newcommand{\T}{\mathens{T}}
\newcommand{\CP}{\C\textup{P}}
\newcommand\sphere[1]{\mathens{S}^{#1}}

\newcommand{\la}{\left\langle}
\newcommand{\ra}{\right\rangle}

\DeclareMathOperator{\dist}{dist}

\DeclareMathOperator{\supp}{supp}

\newcommand{\ham}{\textup{Ham}}

\newcommand\Symp{\textup{Sp}}
\newcommand{\id}{\textup{id}}

\DeclareMathOperator\im{im}
\DeclareMathOperator\ind{ind}
\DeclareMathOperator\mind{\overline{ind}}
\DeclareMathOperator\mas{mas}
\DeclareMathOperator\mmas{\overline{mas}}
\newcommand\comp{\bullet}

\newtheorem{thm}{Theorem}[section]
\newtheorem{lem}[thm]{Lemma}
\newtheorem{cor}[thm]{Corollary}
\newtheorem{prop}[thm]{Proposition}

\newtheorem{prop-def}[thm]{Definition-proposition}

\theoremstyle{definition}

\theoremstyle{remark}

\newcommand\action{\mathcal{T}}
\newcommand\lochom{\textup{C}}
\newcommand\HH{H^*(\{\action\leq 1+\varepsilon'\},
\{\action\leq -\varepsilon'\})}

\begin{document}

\title[Periodic points in $\CP^d$ via generating functions]
{On Periodic points of Hamiltonian diffeomorphisms of
$\CP^d$ via generating functions}

\author[S. Allais]{Simon Allais}
\address{Simon Allais,
\'Ecole Normale Sup\'erieure de Lyon,
UMPA\newline\indent  46 all\'ee d'Italie,
69364 Lyon Cedex 07, France}
\email{simon.allais@ens-lyon.fr}
\urladdr{http://perso.ens-lyon.fr/simon.allais/}
\date{April 5, 2020}

\begin{abstract}
    Inspired by the techniques of Givental and Théret,
    we provide a proof with generating functions 
    of a recent result of Ginzburg-Gürel concerning the
    periodic points of Hamiltonian diffeomorphisms of $\CP^d$.
    For instance, we are able to prove that fixed points of pseudo-rotations
    are isolated as invariant sets or that a Hamiltonian
    diffeomorphism with a hyperbolic fixed point has infinitely
    many periodic points.
\end{abstract}
\maketitle

\section{Introduction}

Let $\CP^d$ be the complex $2d$-dimensional space endowed with
its classical symplectic structure $\omega$, that is $\pi^*\omega=i^*\Omega$ where
$\pi : \sphere{2d+1}\to\CP^d$ is the quotient map,
$i:\sphere{2d+1}\hookrightarrow\C^{d+1}$ is the inclusion map
and $\Omega := \sum_j \ud q_j \wedge\ud p_j$ is the canonical
symplectic form of $\C^{d+1}\simeq \R^{2(d+1)}$.
We are interesting in the study of Hamiltonian diffeomorphisms
of $\CP^d$, which are time-one maps of those vector fields $X_t$
satisfying the Hamilton equations $X_t \lrcorner \omega = \ud h_t$
for some smooth maps $h:[0,1]\times\CP^d\to\R$ called
Hamiltonian maps.
Let $\ham(\CP^d)$ be the set of Hamiltonian diffeomorphisms
of $\CP^d$.
In 1985, Fortune-Weinstein \cite{For85} proved that the Arnol'd conjecture holds
in $\CP^d$: any Hamiltonian diffeomorphism $\varphi\in\ham(\CP^d)$
has at least $d+1$ fixed points.
Inspired by the work of Givental \cite{Giv90},
Théret \cite{The98} used generating functions to reprove Fortune-Weinstein's
theorem.
Given $\varphi\in\ham(\CP^d)$, a $k$-periodic point
$x\in\CP^d$ of $\varphi$ is by definition a fixed point
of the $k$-iterated map $\varphi^k$.
Contrary to aspherical symplectic manifolds like the $2d$-dimensional torus
$\T^{2d}$ endowed with the canonical symplectic form,
the Conley conjecture does not hold in $\CP^d$:
there exists Hamiltonian diffeomorphisms with only finitely many
periodic points.
For instance, one can take a rotation $\rho$ of $\CP^d$ defined by
\begin{equation*}
    \rho([z_1:z_2:\cdots:z_{d+1}]) :=
    \left[e^{2i\pi a_1}z_1 : e^{2i\pi a_2}z_2:\cdots:
    e^{2i\pi a_{d+1}}z_{d+1}\right],
\end{equation*}
with rationally independent coefficients $a_1,\dotsc,a_{d+1}\in\R$.
This is indeed a Hamiltonian diffeomorphism
whose only periodic points are its fixed points:
the projection of the canonical base of $\C^{d+1}$.
Notice that this Hamiltonian diffeomorphism has the minimal number of
periodic points.
A Hamiltonian diffeomorphism of $\CP^d$ which has exactly $d+1$ periodic
points is called a \emph{pseudo-rotation} of $\CP^d$.

In the case $d=1$, $\CP^1\simeq\sphere{2}$ and Hamiltonian diffeomorphisms
are the area preserving diffeomorphisms.
Franks \cite{Fra92,Fra96} proved that every area preserving homeomorphism
has either $2$ or infinitely many periodic points.
Therefore, the only Hamiltonian diffeomorphisms of $\CP^1$
with finitely many periodic points are pseudo-rotations.
In 1994, Hofer-Zehnder \cite[p.263]{HZ94} conjectured a higher-dimensional 
generalization of this result:
every Hamiltonian diffeomorphism of $\CP^d$ has either $d+1$
or infinitely many periodic points
(it was stated for more general symplectic manifolds).
In this direction,
a symplectic proof of Franks result (in the smooth setting) 
was provided by Collier \emph{et al.} \cite{CKRTZ}.
In 2019, Shelukhin \cite{She19} proved a version of Hofer-Zehnder conjecture:
if a Hamiltonian diffeomorphism on a closed monotone symplectic manifold
with semisimple quantum homology (\emph{e.g.} $\CP^d$)
has a finite number of contractible periodic points then the sum of
ranks of the local Floer homologies at its contractible fixed points
is equal to the total dimension of the homology of the manifold
(that is $d+1$ for $\CP^d$).

A compact invariant set $K\subset\CP^d$ of a homeomorphism
$\varphi$ is said to be isolated
if there exists a neighborhood $U$ of $K$ such that, for all $p\in U\setminus K$,
$\varphi^k(p)\not\in U$ for some $k\in\Z$.
A fixed point of a Hamiltonian diffeomorphism is said to be
homologically visible if its local Floer homology is non-trivial.
The purpose of this article is to provide an elementary proof
of the following theorem of Ginzburg-Gürel \cite{GG18}:
\begin{thm}\label{thm:main}
    Every Hamiltonian diffeomorphism of $\CP^d$
    which has a fixed point all of whose iterations
    are homologically visible that is isolated as an invariant
    set has infinitely many periodic points.
\end{thm}
As Ginzburg-Gürel already pointed out, Theorem~\ref{thm:main} has
two important corollaries.
If $x$ is a hyperbolic point then
it is always isolated as an invariant set and
the local cohomology of its iterations has rank 1.
\begin{cor}\label{cor:hyperbolic}
    Every Hamiltonian diffeomorphism of $\CP^d$ with a hyperbolic
    fixed point has infinitely many periodic points.
\end{cor}
In fact, this theorem of Ginzburg-Gürel was originally
proven in \cite{GG14} in a more general setting,
including some complex Grassmannians, $\CP^d\times P^{2k}$
where $P$ is symplectically aspherical and $k\leq d$,
monotone products $\CP^d\times\CP^d$.
We mention that the case of $\CP^d\times\T^{2k}$, when $k\leq d$,
can be deduced as well from our techniques.

In the special case of pseudo-rotations,
every fixed point arises from a min-max principle
and thus has a non-trivial local cohomology.
\begin{cor}\label{cor:pseudorotation}
    Each fixed point of a pseudo-rotation of $\CP^d$
    is not isolated as an invariant set. 
\end{cor}
The original proof of Theorem~\ref{thm:main} 
involves a non-trivial estimate on the energy of 
Floer trajectories leaving a periodic orbit
called Crossing energy theorem by Ginzburg-Gürel
\cite[Theorem~3.1]{GG14} \cite[Theorem~6.1]{GG18}
and proved with a Gromov compactness like theorem
on $J$-holomorphic curves.
The second ingredient of the original proof is quantum homology,
which is defined by means of Gromov-Witten invariants.
Although we closely follow the original argument,
our proof employs only elementary machinery:
Morse theory and classical algebraic topology.
Our main tool is generating functions, which are
finite dimensional versions of the action functional
for Hamiltonian diffeomorphisms of $\R^{2d}$.
Inspired by Théret \cite{The98}, we build a
smooth map $\action : M\to \R$ 
defined on a finite dimensional manifold
$M\subset \R\times\CP^N$.
There is
a correspondence between critical points of $\action$
and capped fixed points of $\varphi$.
With this map, the Crossing energy theorem essentially
boils down to elementary analysis.
If $\zeta\in M$ is a critical point of $\action$
associated to a capped fixed point $\bar{z}$, then
$\lochom^*(\bar{z})$ is by definition the local cohomology of $\zeta$
with integral coefficients.
In this setting, the $\textup{q}$ operator of quantum homology is
mimicked by multiplication by $u^{d+1}\in H^*(M)$
where $u$ is the generator of $H^2(\CP^N)$
(notice that we have a morphism $H^*(\CP^N)\to H^*(M)$, since $M\subset \R\times\CP^N$)
so that we can write the following identity
when every object can be defined:
\begin{equation*}
    \lochom^*(\bar{z}\# A) = u^{d+1}\lochom^*(\bar{z}),
\end{equation*}
where $A$ is the generator of $\pi_2(\C P^d)\simeq\Z$ satisfying $\la[\omega],A\ra=-\pi$
(see Proposition~\ref{prop:globallochom} for the precise statement
and Section~\ref{se:discussion} for a further discussion about it).

Incidentally, we give a new composition formula for generating functions
which is analogous to Chaperon's one \cite{Cha84} but works for
the $\C$-linear identification of the diagonal
$(z,Z)\mapsto (\frac{z+Z}{2},i(z-Z))$.
We also give an alternative way to study the projective join of
a subspace of $\CP^N$ with a projective subspace that
does not involve equivariant cohomology.

\subsection*{Organization of the paper}
In Section~\ref{se:introgf}, we provide the background on generating functions.
In Section~\ref{se:mas}, we provide the background on the Maslov index.
In Section~\ref{se:projectivejoin}, we study the cohomological properties
of the projective join needed to study the action of recapping on the cohomology
of the sublevel sets of $\action$.
In Section~\ref{se:gf}, we show how to use generating functions to provide
a finite dimensional analogue of Floer cohomology in $\CP^d$.
In Section~\ref{se:proof}, we prove Theorem~\ref{thm:main} and
Corollaries~\ref{cor:hyperbolic} and \ref{cor:pseudorotation},
postponing the proof of the Crossing energy theorem.
In Section~\ref{se:crossing}, we prove the Crossing energy theorem in our setting.

\subsection*{Acknowledgments}
I am very grateful to my advisor Marco Mazzucchelli, who
introduced me to this problem and discussed it
with me.

\section{Generating functions} \label{se:introgf}

A generating function for Lagrangian submanifold of $T^*\C^n$ is a smooth function
$F:\C^n\times\C^k\to\R$ such that $0$ is a regular value of the $\C^k$-fiber derivative
$\frac{\partial F}{\partial\xi}$.
The space
\begin{equation}\label{eq:critlocus}
    \Sigma_F := \left\{ (q;\xi)\in \C^n\times\C^k \ |\ \frac{\partial F}{\partial \xi}
    (q;\xi) = 0 \right\}
\end{equation}
is a smooth submanifold with dimension $2n$.
Let $\iota_F : \Sigma_F \to T^*\C^n$ denotes the map
$\iota_F(q;\xi) := (q,\partial_q F(q;\xi))$.
Then $\iota_F$ is a Lagrangian immersion and we say that $F$ generates the
immersed Lagrangian
submanifold $L:=\iota_F(\Sigma_F)$.

A conical generating function of $\C^{2n}\simeq T^*\C^n$
is a $C^1$ map $F:\C^n\times\C^k\to\R$
such that
\begin{enumerate}[ 1.]
    \item $F$ is $S^1$-invariant and 2-homogeneous, that is
        \begin{equation*}
            F(\lambda \zeta) = |\lambda|^2 F(\zeta),\quad
            \forall \lambda\in\C,
            \forall \zeta\in\C^n\times\C^k,
        \end{equation*}
    \item $F$ is smooth in the neighborhood of
        $\Sigma_F \setminus 0$ where the subset $\Sigma_F\subset \C^n\times\C^k$
        is still defined by (\ref{eq:critlocus})
    \item $0$ is a regular value of the fiber derivative $\partial_\xi F$ on
        $\C^n\times\C^k\setminus 0$.
\end{enumerate}
The set $\Sigma_F$ is $\C$-invariant
and so is $\widetilde{L}:=\iota_F(\Sigma_F)$.
If $\pi:\C^{2n}\setminus 0\to \CP^{2n-1}$ denotes the quotient map,
then $L:=\pi(\widetilde{L})$ is a smooth immersed Lagrangian of $\CP^{2n-1}$.
We will say that $\widetilde{L}$ is a conical immersed Lagrangian.

A quadratic generating function $Q:\C^n\times\C^N\to\R$ is
a generating function which is also a quadratic form.
In this case, the induced Lagrangian $\iota_Q(\Sigma_Q)$ is
a linear Lagrangian subspace of $T^*\C^n$.
Notice that if $F:\C^n\times\C^k\to\R$ is a generating function
of the Lagrangian $L\subset \C^{2n}$,
then the quadratic form $\ud^2 F(x):\C^n\times\C^k\to\R$,
for $x\in\Sigma_F$, is a quadratic generating function
of the tangent space $T_{\iota_F(x)} L \subset \C^{2n}$.
The same is true if $F$ is conical and $x\in\Sigma_F\setminus 0$.
Moreover, $\C x \subset \ker \ud^2 F(x)$ in this case.

The existence of generating functions is well known for
Lagrangians which are isotopic to the 0-section $\C^n\times\{ 0\}$
with a ``suitably controlled'' behavior at infinity
(\emph{e.g.} for a compactly supported isotopy
or for a linear isotopy).
In fact, we usually find a generating family of a whole isotopy
$(L_t):=(\Phi_t(\C^n\times\{ 0\}))$, where $(\Phi_t)$ is
a Hamiltonian flow,
that is a continuous family $(F_t)$ of generating functions
with $F_t$ generating $L_t$ for all $t\in[0,1]$.
In Section~\ref{se:gf}, we give a construction of
generating families for Hamiltonian flows.

There are strong uniqueness results relative to generating functions
of linear Lagrangians
or Lagrangians isotopic to the 0-section through compactly supported
isotopies.
Concerning quadratic generating functions, we will only need the following elementary result
\begin{lem}[{\cite[Prop.~35]{ThePHD}}]
    \label{lem:quad0}
    For every quadratic generating function
    $Q:\C^n\times\C^k\to\R$ 
    of the $0$-section, there exists a linear fibered isomorphism
    $A$ of $\C^n\times\C^k$ which is isotopic to the identity through
    linear fiberwise isomorphism such that
    $Q\circ A(q;\xi)$ does not depend on $q\in\C^n$.
    More precisely, if $Q(z)=\la \widetilde{Q}z,z\ra$ with
    \begin{equation*}
        \widetilde{Q}=\begin{bmatrix} a & b\\ {}^t b & c \end{bmatrix},
    \end{equation*}
    then $c$ is invertible and
    $A(q;\xi) := (q;\xi - c^{-1}{}^t bq)$ so that
    $Q\circ A(q;\xi) = {}^t\xi c \xi$.
\end{lem}
Concerning the conical case, we will use the following
\begin{lem}[{\cite[Lemma~4.8]{The98}}]
    \label{lem:isotopy}
    If $(F_t : \C^n\times\C^k\to\R)$ is a smooth family of conical
    generating functions for the constant conical Lagrangian
    $L\subset T^*\C^n$, then there is a smooth isotopy
    $(B_t)$ of conical fibered diffeomorphisms such that
    \begin{equation*}
        F_t\circ B_t = F_0,\quad \forall t.
    \end{equation*}
\end{lem}

Let $\ham(\C^d)$ be the set of Hamiltonian diffeomorphisms of
$\C^d\simeq T^*\R^d$.
The map
\begin{equation}
    \tau : \overline{\C^d}\times\C^d \to \C^{2d},\quad
    \tau(z,Z) = \left(\frac{z+Z}{2},i(z-Z)\right),
\end{equation}
is a $\C$-linear symplectomorphism sending the diagonal
$\{ (z,z)\ |\ z\in\C^d\}$ to the 0-section of $\C^{2d}$.
Let $\Phi\in\ham(\C^d)$,
the image of the graph $z\mapsto (z,\Phi(z))$ of $\Phi$ under
$\tau$ is then a Lagrangian submanifold $L_\Phi\subset \C^{2d}$.
A generating function of the Hamiltonian diffeomorphism $\Phi$ is
a generating function of $L_\Phi$.
A generating family of a Hamiltonian flow $(\Phi_t)$
is a generating family of $(L_{\Phi_t})$.
Let $\Phi$ be a conical Hamiltonian diffeomorphism of $\C^d$,
that is, a homeomorphism $\Phi:\C^d\to\C^d$ 
with $\Phi|_{\C^d\setminus 0} \in\ham(\C^d\setminus 0)$
and which is
$\C$-equivariant:
\begin{equation*}
    \Phi(\lambda z) = \lambda \Phi(z),\quad
    \forall \lambda\in\C,\forall z\in\C^n.
\end{equation*}
To simplify notation, we will write $\Phi\in\ham_\C (\C^d)$ and say
that $\Phi$ is a $\C$-equivariant Hamiltonian diffeomorphism.
The last definition extends to Hamiltonian flow in the obvious way.
Then the induced subset $L_\Phi\subset \C^{2d}$ is a conical Lagrangian.
A conical generating function of $\Phi$
(or simply a generating function of $\Phi$)
is a conical generating function of $L_\Phi$.
It extends to conical flow in the obvious way.
As a consequence of the general case,
if $F$ is a generating function of $\Phi\in\ham(\C^d)$
and $(z;\xi)$ is a critical point of $F$ then
$z$ is a fixed point of $\Phi$ and
$\ud^2 F(z;\xi)$ is a quadratic generating function of
$\ud\Phi(z)$.
Moreover, 
\begin{equation*}
    \dim\ker\ud^2 F(z;\xi) = \dim\ker(\ud\Phi(z)-\id).
\end{equation*}

\section{Maslov Index} \label{se:mas}

\subsection{Maslov index of a path in $\Symp(2d)$}

Let $\Gamma = (\Gamma_t) : [0,1]\to\Symp(2d)$ be a continuous path in
the space of symplectic matrices $\Symp(2d)$ of $\R^{2d}\simeq\C^d$.
Then there exists a continuous family $(Q_t)$ of quadratic generating functions
such that, for $t\in[0,1]$,
$Q_t : \C^N \to \R$ is generating $\Gamma_t$.
The variation of index $\ind(Q_1)-\ind(Q_0)\in\Z$ is independent of the choice
of $(Q_t)$ and is called the Maslov index of $\Gamma$ denoted
\begin{equation*}
    \mas((\Gamma_t)) :=
    \ind(Q_1) - \ind(Q_0)\in\Z.
\end{equation*}
Other equivalent definitions of the Maslov index
(which is sometimes also called Conley-Zehnder index)
are available in the literature, see \cite{SZ92},
\cite{Lon02} and references therein.

In order to state the general properties of $\mas$,
following Théret, in this section we will denote
by $R\comp S$ the concatenation of two paths
$R=(R_t)$ and $S=(S_t)$ in $\Symp(2d)$ satisfying $R_1=S_0$,
that is $(R\comp S)_t = R_{2t}$ for $t\in[0,1/2]$
and $(R\comp S)_t = S_{2t-1}$ for $t\in[1/2,1]$.
The path $RS$ stands for the pointwise matrix product of two
paths in $\Symp(2d)$ that is $(RS)_t = R_t S_t$ for all $t$.
Given a path $R=(R_t)$ in $\Symp(2d)$,
the path $R^{(-1)}$ will stand for the reverse path
$(R_{1-t})$, whereas $R^{-1}$ will stand for the path of inverses
$(R^{-1}_t)$.
Identifying matrices with their canonical linear maps,
for two square matrices $A$ and $B$,
$A\oplus B$ will stand for the square matrix
\begin{equation*}\begin{bmatrix} A&0\\ 0&B \end{bmatrix}\end{equation*}
and given two paths $R=(R_t)$ and $S=(S_t)$ in
$\Symp(2n)$ and $\Symp(2m)$ respectively,
$(R\oplus S)_t := (R_t\oplus S_t)$ as a path in $\Symp(2(n+m))$.
We recall the basic proprieties of the Maslov index
(see for instance \cite[Prop.~39 and 60]{ThePHD}).

\begin{prop}
    \label{prop:maspath}
    Let $R$ be a path in $\Symp(2n)$,
    \begin{enumerate}[(1)]
        \item\label{it:comp}
            if $S$ is a path in $\Symp(2n)$ with
            $S_0 = R_1$, then
            $\mas(R\comp S) = \mas(R) +\mas(S)$,
        \item\label{it:rev}
            the Maslov index of the reverse path is
            $\mas(R^{(-1)}) = -\mas(R)$,
        \item\label{it:sum}
            if $S$ is a path in $\Symp(2m)$, then
            $\mas(R\oplus S) = \mas(R) + \mas(S)$,
        \item\label{it:symp}
            if $A\in\Symp(2d)$,
            then
            $\mas(ARA^{-1}) = \mas(R)$.
        \item\label{it:homot}
            if $S$ is a path homotopic to $R$
            relative to endpoints,
            that is there exists a continuous family
            $s\mapsto R^s$ of paths in $\Symp(2n)$
            with $R^0 = R$ and $R^1 = S$
            such that $R^s_0 \equiv R^0_0$
            and $R^s_1 \equiv R^0_1$, then
            $\mas(S) = \mas(R)$,
        \item\label{it:rot}
            if $S_t := 
            \begin{bmatrix} \cos(2\pi t)& -\sin(2\pi t)\\
            \sin(2\pi t) & \cos(2\pi t) \end{bmatrix}\in\Symp(2)$,
            $t\in[0,1]$,
            then
            $\mas(S) = -2$.
    \end{enumerate}
\end{prop}

Let $(\Phi_t)$ be a Hamiltonian flow on $\C^d$ starting at
$\Phi_0 = \id$.
If $z\in\C^d$ is a fixed-point of $\Phi_1$,
the Maslov index of $z$ is set to be the Maslov index of the path
$t\mapsto\ud\Phi_t(z)$ in $\Symp(2d)$, that is
\begin{equation*}\label{eq:masham}
    \mas(z,(\Phi_t)) := \mas((\ud\Phi_t(z))).
\end{equation*}
Suppose that $F_t:\C^N\to\C$, $t\in[0,1]$,
defines a continuous family of generating functions of $(\Phi_t)$.
Let $\zeta_t\in\Sigma_{F_t}\subset\C^N$ be a continuous family
associated to $\Phi_t(z)$.
Then the continuous family of Hessian $Q_t := \ud^2 F_t(\zeta_t)$
is a continuous family of quadratic generating functions of $\ud\Phi_t(z)$,
thus
\begin{equation*}
    \mas(z,(\Phi_t)) = \ind(\zeta_1,F_1) - \ind(\zeta_0,F_0),
\end{equation*}
where $\ind(\zeta,F):=\ind(\ud^2 F(\zeta))\in\N$ denotes the Morse
index of $F$ at the critical point $\zeta$.

This definition is extended to every symplectic manifold $M^{2d}$
as follows.
Let $(\varphi_t)$ be a Hamiltonian flow on $M^{2d}$ starting at $\varphi_0 =\id$
and let $z\in M$ be a fixed point of $\varphi_1$
such that the loop $t\mapsto\varphi_t(z)$ is contractible.
Let $D^2:=\{ w\in\C \ |\ |w|\leq 1\}$ be the closed unit disk of $\C$.
Since the loop is contractible,
there exists a smooth map $u:D^2\to M$ such that $u(e^{2i\pi t}) = \varphi_t(z)$.
Then there exists a trivialization $D^2\times \C^d\to u^*TM$,
$(w,\zeta)\mapsto \xi(w)\zeta$ so that,
for all $w\in D^2$, $\xi(w):\C^d\to T_{u(w)}M$
is a symplectic map.
Moreover, if we endow $M$ with an almost complex structure,
the trivialization can be made $\C$-linear.
The set of every such trivialization is contractible, for a fixed choice of
$u$ (see \cite[Lemma~5.1]{SZ92} for instance).
Then $\gamma_t := \xi(e^{2i\pi t})^{-1}\ud\varphi_t(z)\xi(1)$, $t\in[0,1]$,
is a symplectic path in $\Symp(2d)$ and
the Maslov index of $z$ with respect to the \emph{capping} $u$ is set to be
\begin{equation*}\label{eq:mashamgen}
    \mas(z,u) := \mas((\gamma_t)).
\end{equation*}
It does not depend on the specific choice of trivialization,
in fact it only depends on the homotopy class of $u$
relative to the boundary $\partial D^2$.
Thus, if $\pi_2(M) = 0$ any choice of $u$ gives the same index.

\subsection{Maslov index of a $\C$-equivariant Hamiltonian diffeomorphism}
\label{se:masequiv}

Let $(\Phi_t)$ be a $\C$-equivariant Hamiltonian flow on $\C^{d+1}$ lifting
a Hamiltonian flow $(\varphi_t)$ on $\CP^d$.
Let $Z_0\in\sphere{2d+1}$ be a fixed point of $\Phi_1$
and denote by $Z_t := \Phi_t(Z_0)$, $t\in[0,1]$,
the associated loop in $\sphere{2d+1}$.
Let $\pi:\sphere{2d+1}\to\CP^d$ be the quotient map.
Let $z_t :=\pi(Z_t)$ be the associated loop in $\CP^d$
so that $z_t = \varphi_t(z_0)$.
Let $U:D^2\to\sphere{2d+1}$ be any smooth capping of $(Z_t)$,
\emph{i.e.} $Z_t = U(e^{2i\pi t})$.
All such cappings are homotopic since $\pi_2(\sphere{2d+1})=0$.
We set $u:=\pi\circ U$.

\begin{prop}\label{prop:maslift}
    With the above notations,
    \begin{equation*}
        \mas(Z_0,(\Phi_t)) = \mas(z_0,u).
    \end{equation*}
\end{prop}

\begin{proof}
For all $t\in[0,1]$,
let $\gamma_t := \ud\varphi_t(z_0) : T_{z_0}\CP^d\to T_{z_t}\C P^d$
and $\Gamma_t := \ud\Phi_t(Z_0)$ which is a path in $\Symp(2(d+1))$.
For all $w\in D^2$, let $\xi(w):\C^d\to T_{u(w)}\CP^d$
be a smooth family of $\C$-linear symplectic map induced by
$u$ as explained above.
Throughout the proof, if $f$ denotes a map whose domain is $D^2$,
then, for $t\in[0,1]$, $f_t:=f(e^{2i\pi t})$.
For all $t\in[0,1]$
let $\xi_t := \xi(e^{2i\pi t})$ and
$\gamma'_t := \xi_t^{-1}\gamma_t\xi_0\in\Symp(2d)$ so that
\begin{equation*}
    \mas(Z_0,(\Phi_t)) = \mas((\Gamma_t))\quad
    \text{and}\quad
    \mas(z_0,u) = \mas((\gamma'_t)).
\end{equation*}
Notice that, for all $Z\in\sphere{2d+1}$, the tangent space
$T_{\pi(Z)}\CP^d \simeq \C^{d+1}/\C Z$ is canonically
isomorphic to $(\C Z)^\bot$
(given a $\C$-subspace $E\subset\C^{d+1}$,
$E^\bot$ denotes its hermitian orthogonal subspace,
which is also its Euclidean orthogonal subspace
or its symplectic orthogonal subspace).
Let $L(w):=(\C U(w))^\bot \to T_{u(w)}\CP^d$, $w\in D^2$,
be the induced continuous family of $\C$-linear symplectic maps.
Let us define the following continuous family of endomorphism of
$\C^{d+1}$ indexed by $w\in D^2$,
\begin{equation*}
    A(w) : \C\times\C^d\to\C U(w)\oplus(\C U(w))^\bot,\quad
    A(w)(\lambda,\zeta) = \lambda U(w) + L(w)^{-1}\xi(w)\zeta.
\end{equation*}
Since the linear maps $\lambda\mapsto \lambda U(w)$
and $L(w)^{-1}\xi(w)$ are symplectic maps
and since both direct sums $\C\times\C^d$
and $\C U(w)\oplus (\C U(w))^\bot$ 
are symplectic-orthogonal sums,
$A(w)\in\Symp(2(d+1))$.

Since $\Phi_t$ is a $\C$-equivariant diffeomorphism,
The symplectic map $\ud \Phi_t(Z_0)=\Gamma_t$ sends the orthogonal subspaces
$\C Z_0$ and $(\C Z_0)^\bot$ respectively on
$\C Z_t$ and $(\C Z_t)^\bot$ with
\begin{equation*}
    \Gamma_t(\lambda Z_0 + \zeta) =
    \lambda Z_t + L_t^{-1}\gamma_t L_0 \zeta,\quad
    \forall\lambda\in\C,\forall\zeta\in(\C Z_0)^\bot,
\end{equation*}
where $L_t := L(e^{2i\pi t}) : (\C Z_t)^\bot \to T_{z_t}\CP^d$.
Thus $\Gamma'_t := A_t^{-1}\Gamma_t A_0$ is the symplectic path
$\Gamma'_t = I_2 \oplus \gamma'_t$, so
Proposition~\ref{prop:maspath}~(\ref{it:sum}) implies
$\mas((\Gamma'_t)) = \mas((\gamma'_t))$.
Since $A_t = A(e^{2i\pi t})$ with $A:D^2\to\Symp(2(d+1))$
continuous,
$(\Gamma'_t)$ is homotopic to $(A_0^{-1}\Gamma_t A_0)$ relative
to endpoints,
thus $\mas((\Gamma'_t))=\mas((A_0^{-1}\Gamma_t A_0))
=\mas((\Gamma_t))$,
according to Proposition~\ref{prop:maspath}~(\ref{it:homot})
and (\ref{it:symp}).
\end{proof}

\subsection{Bott iteration inequalities}

Let $(\Phi_t)$ be a Hamiltonian flow on $\C^d$ starting at
$\Phi_0 = \id$ and let $z\in\C^d$ be a fixed point.
Even though $(\Phi_t(z))$ is a loop in $\C^d$,
$\Gamma_t:=\ud\Phi_t(z)$, $t\in\R_+$, defines only 
a path in $\Symp(2d)$,
so that in general $\mas(\Gamma_{kt},t\in[0,1])\neq
k\mas(\Gamma_t,t\in[0,1])$.
Notice that the path $(\Gamma_t)_{t\in\R}$ only depends on
$(\Gamma_t)_{t\in[0,1]}$ since
$\Gamma_{t+k} = \Gamma_t\Gamma_1^k$ for $k\in\N$ and $t\geq 0$.

\begin{thm}\label{thm:mmas}
    Let $\Gamma :=(\Gamma_t)_{t\geq 0}$ be a continuous path in $\Symp(2d)$
    such that $\Gamma_0 = I_{2d}$ and $\Gamma_{t+k}=\Gamma_t\Gamma_1^k$
    for all $k\in\N$ and $t>0$.
    Then the \emph{average Maslov index}
    \begin{equation*}
        \mmas(\Gamma) := \lim_{k\to\infty}\frac{\mas(\Gamma_{kt},t\in[0,1])}{k}\in\R
    \end{equation*}
    is a well defined real number
    and we have the iteration inequalities
    \begin{equation*}
        \begin{gathered}
        k\mmas(\Gamma) - d \leq \mas(\Gamma_{kt},t\in[0,1]),\\
        \mas(\Gamma_{kt},t\in[0,1]) + \dim\ker(\Gamma_1^k - I_{2d})
        \leq  k\mmas(\Gamma) +d.
        \end{gathered}
    \end{equation*}
\end{thm}

We refer to \cite[Theorem~3.6]{Maz16} for a more precise statement
and a proof.
Notice that, by definition, the average Maslov index is homogeneous:
\begin{equation*}\label{eq:mmashom} 
    \mmas((\Gamma_{kt}))=k\mmas((\Gamma_t)).
\end{equation*}
Let us denote by $\mmas(z,(\Phi_t))\in\R$ the average Maslov index
of the fixed point $z$, that is 
\begin{equation*}
\mmas(z,(\Phi_t)):=\mmas(\ud\Phi_t(z),t\geq 0).
\end{equation*}
So that Theorem~\ref{thm:mmas} gives for all $k\in\N$,
\begin{equation*}
        \begin{gathered}
            k\mmas(z,(\Phi_t)) - d \leq \mas(z,(\Phi_{kt})),\\
            \mas(z,(\Phi_{kt})) + \dim\ker(\ud\Phi_k(z)-\id)
            \leq  k\mmas(z,(\Phi_t)) +d.
        \end{gathered}
\end{equation*}
This inequality can be extended to every symplectic manifold $M^{2d}$ as follows.
Let $(\varphi_t)$ be a Hamiltonian flow on $M^{2d}$ starting at $\varphi_0=\id$
and let $z\in M$ be a fixed point of $\varphi_1$ such that $(\varphi_t(z))$
is contractible.
Let $u:D^2\to M$ be a capping of $z$ and $\xi(w):\C^d\to T_{u(w)}M$,
$w\in D^2$, be an induced trivialization.
For $k\in\N^*$, let $u^k:D^2\to M$ be the smooth map
$u^k(w):=u(w^k)$, $w\in D^2$.
This map is the natural capping of $z$ as a fixed point of the time-one map
of the Hamiltonian flow $(\varphi_{kt})$ induced by $(z,u)$.
If $\bar{z}:=(z,u)$, it is often denoted by $\bar{z}^k = (z,u^k)$.
An induced trivialization is $\xi^k(w):=\xi(w^k)$,
so that $\gamma^{(k)}_t = \gamma_{kt}$,
where $\gamma^{(k)}_t := \xi^k(e^{2i\pi t})^{-1}\ud\varphi_{kt}(z)\xi^k(1)$,
and $\gamma_t := \xi(e^{2i\pi t})^{-1}\ud\varphi_t(z)\xi(1)$,
$t\geq 1$.
Since $\mas(\bar{z}^k):=\mas(\gamma^{(k)}_t,t\in[0,1])$
with $\gamma_{t+k}=\gamma_t\gamma_1^k$ for all $k\in\N$ and $t\geq 0$,
Theorem~\ref{thm:mmas} gives for all $k\in\N$,
\begin{equation}\label{eq:mmascapped}
    \begin{gathered}
        k\mmas(\bar{z}) - d \leq \mas(\bar{z}^k),\\
        \mas(\bar{z}^k) + \dim\ker(\ud\varphi_k(z)-\id)
        \leq  k\mmas(\bar{z}) +d.
        \end{gathered}
\end{equation}
where $\mmas(\bar{z}) := \mmas(\gamma_t,t\geq 0)$ is the average Maslov
index of the capped fixed point $\bar{z}=(z,u)$.
Let $(\Phi_t)$ be a $\C$-equivariant Hamiltonian flow of $\C^{d+1}$ with $\Phi_0=\id$ which
is the lift of a Hamiltonian flow $(\varphi_t)$ of $\CP^d$
with $\varphi_0 =\id$.
Let $Z\in\sphere{2d+1}$ be a fixed point of $\Phi_1$
and $\bar{z}=(\pi(Z),u)$ be the capped fixed point of $\varphi_1$
associated to it, then
\begin{equation*}
    \mas(Z,(\Phi_{kt})) = \mas(\bar{z}^k),
        \quad\forall k\in\N.
\end{equation*}
Indeed, if $U:D^2\to\sphere{2d+1}$ is a capping of $Z$
so that $u=\pi\circ U$, then $U^k$ is a capping of $Z$
relative to $(\Phi_{kt})_{t\in[0,1]}$
and $u^k=\pi\circ U^k$
(recall that $\mas(Z,(\Phi_{kt}))$ does not depend
on the choice of capping since $\pi_2(\sphere{2d+1})=0$).
So, according to equation (\ref{eq:mmascapped}), for every fixed point $Z\in\sphere{2d+1}$
of any $\C$-equivariant Hamiltonian flow $(\Phi_t)$ of $\C^{d+1}$
which is the lift of some Hamiltonian flow $(\varphi_t)$ of $\CP^d$,
for all $k\in\N$,
\begin{equation}\label{eq:mmasequiv}
    \begin{gathered}
            k\mmas(Z,(\Phi_t)) - d \leq \mas(Z,(\Phi_{kt})),\\
            \mas(Z,(\Phi_{kt})) + \dim\ker(\ud\varphi_k(z)-\id)
            \leq  k\mmas(Z,(\Phi_t)) +d,
    \end{gathered}
\end{equation}
where $z:=\pi(Z)$.

\section{Projective join}\label{se:projectivejoin}

In \cite[Appendix]{Giv90}, Givental studied the cohomology of
projective joins by using $S^1$-equivariant cohomology.
Here, we give an alternative way to study the special case of joins with
a projective subspace.

Let $m,n\in\N$ and let $\pi:\C^{m+n+2}\setminus 0\to \CP^{m+n+1}$ be the
quotient projection. We projectively embed $\CP^m$ and $\CP^n$ in $\CP^{m+n+1}$
by identifying $\CP^m$ with $\pi(\C^{m+1}\times 0\setminus 0)$ and $\CP^n$ with
$\pi(0\times \C^{n+1}\setminus 0)$ so that $\CP^n$ and $\CP^m$ do not
intersect.  This is equivalent to considering two projective subspaces of
respective $\C$-dimensions $m$ and $n$ in general position.  Let $A\subset
\CP^m$ and $B\subset \CP^n$. Then the projective join $A*B\subset \CP^{m+n+1}$
is the union of every projective lines intersecting $A$ and $B$.  In other
words, $A*B = A\cup B\cup \pi(\widetilde{A}\times\widetilde{B})$ where $\widetilde{A}$ and
$\widetilde{B}$ are the lifts of $A$ and $B$ to $\C^{m+1}\setminus 0$ and
$\C^{n+1}\setminus 0$ respectively.  One can remark that $\CP^m *\CP^n =
\CP^{m+n+1}$ and that if $[a:b]\in\CP^{m+n+1}$, with $a\in\C^{m+1}$ and $b\in\C^{n+1}$,
does not belong to $\CP^m$ nor to $\CP^n$, then only one projective line
intersecting $\CP^m$ and $\CP^n$ contains $[a:b]$, namely the line joining $[a:0]$
to $[0:b]$.
Given $A\subset \CP^m$, we denote by $p_A : A*\CP^n\setminus \CP^n\to A$
the projection $[a:b]\mapsto [a:0]$.

In this paper, $H^*$ will stand for the singular cohomology with integer coefficients.
Given $A\subset \CP^m$, let $T\subset A*\CP^m$ be a tubular neighborhood of $\CP^m$
such that $(A*\CP^n,T)$ retracts on $(A*\CP^n,\CP^n)$.
By excision $H^*(A*\CP^n,\CP^n)\simeq H^*(A*\CP^n\setminus \CP^n,T\setminus \CP^n)$.
Using this identification, we define the cup-product 
$H^*(A*\CP^n\setminus\CP^n)\otimes H^*(A*\CP^n,\CP^n)
\to H^*(A*\CP^n,\CP^n)$ by the following commutative diagram
\begin{equation*}
    \begin{gathered}
        \resizebox{.9\displaywidth}{!}{
            \xymatrix{
                H^*(A*\CP^n\setminus\CP^n)\otimes H^*(A*\CP^n,\CP^n) \ar[r]^-{\smile} \ar[d]^-{\simeq} &
                H^*(A*\CP^n,\CP^n) \ar[d]^-{\simeq} \\
                H^*(A*\CP^n\setminus\CP^n)\otimes H^*(A*\CP^n,T) & H^*(A*\CP^n,T) \\
                H^*(A*\CP^n\setminus\CP^n)\otimes H^*(A*\CP^n\setminus \CP^n,T\setminus\CP^n)
                \ar[u]^-{\simeq} \ar[r]^-{\smile} & H^*(A*\CP^n\setminus \CP^n,T\setminus\CP^n)
                \ar[u]^-{\simeq}
        } }
    \end{gathered},
\end{equation*}
where the vertical arrows are induced by inclusions and the bottom arrow is
the usual cup-product.
According to the long exact sequence of the couple $(\CP^{m+n+1},\CP^n)$,
the map $H^{2(n+1)}(\CP^{m+n+1},\CP^n)\to H^{2(n+1)}(\CP^{m+n+1})$
induced by the inclusion is an isomorphism (the dimension of $\CP^n$
being $2n<2n+1$) so that we can see the class $u^{n+1}\in H^{2(n+1)}(\CP^{m+n+1})$
in $H^{2(n+1)}(\CP^{m+n+1},\CP^n)$ \emph{via} this identification.
Given $A\subset\CP^m$, let $t_A\in H^{2(n+1)}(A*\CP^n,\CP^n)$
be the image of $u^{n+1}\in H^{2(n+1)}(\CP^{m+n+1},\CP^n)$
induced by the inclusion and let 
$f_A:H^*(A)\to H^{*+2(n+1)}(A*\CP^n)$ be the morphism given by
$f_A(v) := p_A^*(v) \smile t_A$.

\begin{prop}\label{prop:join}
    Let $A\subset\CP^m$.
    One has the following isomorphisms:
    \begin{equation*}
        H^k(A*\CP^n) \simeq 
        \begin{cases}
            H^k(\CP^n) & \text{for } k\leq 2n+1, \\
            H^{k-2(n+1)}(A) & \text{for } k>2n+1,
        \end{cases}
    \end{equation*}
    where the isomorphisms $H^k(A *\CP^n)\to H^k(\CP^n)$ are induced
    by the inclusion and the isomorphisms $H^{k-2(n+1)}(A)\to H^k(A*\CP^n)$ are given by $f_A$.
\end{prop}

\begin{proof}
    Let us consider the long exact sequence of the couple $(A*\CP^n,\CP^n)$:
    \begin{equation}\label{eq:lesACPn}
        \cdots \to H^*(A*\CP^n,\CP^n)\xrightarrow{j^*}
        H^*(A*\CP^n) \xrightarrow{i^*} H^*(\CP^n) \to \cdots
    \end{equation}
    The inclusions of $A*\CP^n$ and $\CP^n$ in $\CP^{m+n+1}$ give the
    following commutative diagram:
    \begin{equation*}
        \xymatrix{ H^*(A *\CP^n) \ar[r]^-{i^*} & H^*(\CP^n) \\
            H^*(\CP^{m+n+1}) \ar[u] \ar[ur]
        }
    \end{equation*}
    where the diagonal arrow is onto (we recall that $\CP^n$
    is projectively embedded inside $\CP^{m+n+1}$), thus $i^*$ is onto.
    Hence the long exact sequence (\ref{eq:lesACPn}) can be reduced to the
    short exact sequence
    \begin{equation}\label{eq:sesACPn}
        0 \to H^*(A*\CP^n,\CP^n)\xrightarrow{j^*}
        H^*(A*\CP^n) \xrightarrow{i^*} H^*(\CP^n) \to 0.
    \end{equation}
    
    Let us consider $p_A : A*\CP^n\setminus\CP^n\to A$.
    This projection defines a complex vector bundle of dimension $n+1$.
    Indeed, let $E_A := A*\CP^n\setminus\CP^n$ and $U_i \subset \CP^{m+n+1}$
    be the affine chart $\{ [a_0:\cdots:a_m:z_0:\cdots:z_n]\ |\ a_i\neq 0\}$.
    Since the intersection of a projective line with the projective hyperplane
    $\CP^{m+n+1}\setminus U_i$ is either a point or the projective line itself,
    we see that $p_A^{-1}(A\cap U_i) = E_A\cap U_i$.
    We then have the trivialization $E_A\cap U_i \simeq A\cap U_i \times \C^{n+1}$
    given by $[a:z]\mapsto ([a],z/a_i)$.
    Thus $E_A$ is a fiber bundle, moreover this is the restriction of $E_{\CP^m}$
    to $A$. We can even say that $E_{\CP^m} \simeq (\gamma^1_m)^{\oplus(n+1)}$
    where $\gamma^1_m$ is the tautological fiber bundle of $\CP^m$,
    by looking at the transition maps of the above trivialization charts
    (but this will not be relevant for us).
    Let us endow $\CP^{m+n+1}$ with the Riemannian metric induced by the round
    metric of $\sphere{2(m+n)+3}$ and let $T\subset A*\CP^n$ be the tubular neighborhood
    of $\CP^n$ defined as the set of points at distance less than $r\in (0,\pi/2)$
    of $\CP^n$. Then the topological pair $(A*\CP^n,T)$ retracts on $(A*\CP^n,\CP^n)$
    so that the inclusion map induces an isomorphism
    $H^*(A*\CP^n,\CP^n)\simeq H^*(A*\CP^n,T)$ in cohomology.
    Since the compact $\CP^n$ is included in the interior of $T$,
    by excision $H^*(A*\CP^n,T)\simeq H^*(E_A,T\cap E_A)$.
    In the trivialization charts, each fibers of $E_A\setminus T$ is a round
    ball of $\C^{n+1}$ so that $(E_A,E_A\setminus A)$ retracts on
    $(E_A,T\cap E_A)$.
    According to Thom isomorphism theorem,
    \begin{equation*}
        H^{*-2(n+1)}(A) \simeq H^*(E_A,E_A\setminus A) \simeq H^*(A*\CP^n,\CP^n),
    \end{equation*}
    where the isomorphism $H^{*-2(n+1)}(A)\to H^*(A*\CP^n,\CP^n)$
    is given by the cup-product of the pull-back of the class by $p_A$
    with the Thom class $t'_A\in H^{2(n+1)}(A*\CP^n,\CP^n)$.
    Furthermore, since $H^k(\CP^n)$ is zero when $k>2n$ and
    $H^k(A*\CP^n,\CP^n)$ is zero when $k<2(n+1)$, the short exact sequence
    (\ref{eq:sesACPn}) obviously decomposes: $H^*(A*\CP^n)\simeq
    H^*(A*\CP^n,\CP^n)\oplus H^*(\CP^n)$.

    Since $E_A$ is the restriction of $E_{\CP^m}$, the Thom class $t'_A$
    is the image of the Thom class $t'_{\CP^m}$ under the morphism
    induced by inclusion. 
    Since $j^*$ must be an isomorphism in degree $2(n+1)$
    in the exact sequence (\ref{eq:sesACPn}) for $A=\CP^m$,
    we must have $t'_{\CP^m} = \pm u^{n+1}$ (recall that
    $\CP^m *\CP^n = \CP^{m+n+1}$).
    In fact $t'_{\CP^m} = u^{n+1} = t_{\CP^m}$ as the orientation of a complex fiber $\simeq \C^{n+1}$
    coincides with the orientation of a projective subspace of $\C$-dimension
    $n+1$ (they all come from the complex structure of $\CP^{m+n+1}$).
\end{proof}

Following Givental, we define $\ell(A)\in\N$ for $A\subset\CP^N$
as the rank of the morphism $H^*(\CP^N)\to H^*(A)$ induced by the inclusion
(\emph{e.g.} $\ell(\CP^n) = n+1$).
This definition coincides with the equivariant cohomological index
defined by Fadell and Rabinowitz \cite{FR78} (in the special case of the free action of $S^1$
on $\sphere{2N+1}$).

\begin{cor}[{\cite[Corollary~A.2]{Giv90}}]\label{cor:homlength}
    Let $A\subset\CP^m$, then $\ell(A*\CP^n) = \ell(A)+n+1$.
\end{cor}

\begin{proof}
    Since $f_{\CP^m}(u^k)=u^{n+1}\smile u^k$ for
    $0\leq k\leq m$,
    we have the following commutative diagram:
    \begin{equation*}
        \xymatrixcolsep{5pc}
        \xymatrix{
            H^*(\CP^{m+n+1}) \ar[r]^-{u^{n+1}\smile\cdot} \ar[d] &
            H^{*+2(n+1)}(\CP^{m+n+1}) \ar[d] \\
            H^*(A) \ar[r]^-{f_A} & H^{*+2(n+1)}(A*\CP^n)
        }
    \end{equation*}
    where the vertical arrows are induced by inclusions.
    For the grading $*=2(\ell(A*\CP^n)-n-1)$,
    the map $u^{n+1}\smile\cdot$ is onto, so $\ell(A*\CP^n)\leq \ell(A)+n+1$.
    According to Proposition~\ref{prop:stabilization},
    the map $f_A$ is an injection
    for the grading $*=2\ell(A)$, so $\ell(A*\CP^n)\geq \ell(A)+n+1$.
\end{proof}

\section{Generating functions of $\C$-equivariant Hamiltonian diffeomorphism}
\label{se:gf}

\subsection{``Broken trajectories'' and generating functions of $\C^d$}

Let $\Phi\in\ham(\C^d)$ be a Hamiltonian diffeomorphism
which can be decomposed in $\Phi = \sigma_n\circ\cdots\circ\sigma_1$
where every $\sigma_k\in\ham(\C^d)$ is sufficiently $C^1$-close
to $\id$ such that they admit generating functions
$f_k:\C^d\to\R$ satisfying:
\begin{equation}\label{eq:egf}
    \forall z_k\in\C^d, \exists ! w_k\in\C^d,\quad
    w_k=\frac{z_k+\sigma_k(z_k)}{2} \quad \text{ and } \quad
    \nabla f_k (w_k) = i(z_k-\sigma_k(z_k)).
\end{equation}
We call such generating functions without auxiliary variable
elementary generating functions.
We will say that the $n$-tuple $\boldsymbol{\sigma}=(\sigma_1,\dotsc,\sigma_n)$
is associated to the Hamiltonian flow $(\Phi_t)$ if
there exists real numbers $0=t_0\leq t_1\leq \cdots\leq t_n = 1$
such that $\sigma_k = \Phi_{t_k}\circ\Phi_{t_{k-1}}^{-1}$.
A continuous family of such tuples $(\sigma_s)$ will denote
a family of tuples of the same size $n\geq 1$,
$\boldsymbol{\sigma}_s =: (\sigma_{1,s},\dotsc,\sigma_{n,s})$ such that
the maps $s\mapsto \sigma_{k,s}$ are $C^1$-continuous.
Every compactly supported Hamiltonian flow and every $\C$-equivariant
Hamiltonian flow $(\Phi_s)_{s\in[0,1]}$ admit a continuous family of associated tuple
$(\boldsymbol{\sigma}_s)$ that is $\boldsymbol{\sigma}_s$ is associated to $\Phi_s$
for all $s\in[0,1]$ (and the size can be taken as large as wanted).

Denote by $F_{\boldsymbol{\sigma}}$ the following function $(\C^d)^n\to\R$:
\begin{equation}
    F_{\boldsymbol{\sigma}} (v_1,\dotsc,v_n) := \sum_{k=1}^{n}
        f_k\left(\frac{v_k + v_{k+1}}{2}\right) +
    \frac{1}{2}\la v_k,iv_{k+1}\ra,
\end{equation}
with convention $v_{n+1} = v_1$.
Let $A:(\C^d)^n\to(\C^d)^n$ denotes the linear map
such that, for $\mathbf{v}=(v_1,\ldots,v_n)$,
$A(\mathbf{v})=\mathbf{w}$ with $w_k = \frac{v_k+v_{k+1}}{2}$.
Let $\psi:(\C^d)^n\to(\C^d)^n$ be
the diffeomorphism $\psi(\mathbf{z}) = \mathbf{w}$ defined by (\ref{eq:egf}).

\begin{prop}\label{prop:gf}
    Under the above hypothesis, 
    we have 
    \begin{equation*}
        \forall k, \forall \mathbf{v}\in(\C^d)^n, \quad
        \partial_{v_k} F_{\boldsymbol{\sigma}} (v_1,\dotsc,v_n) = i(z_{k} -
        \sigma_{k-1}(z_{k-1})),
    \end{equation*}
    where $\mathbf{z} := \psi^{-1}\circ A(\mathbf{v})$ and $z_{0}:=z_n$.
    Moreover, if $n$ is odd, $F_{\boldsymbol{\sigma}}$ is a generating
    function of $\Phi$ with $v_1$ as main variable.
\end{prop}

\begin{proof}
    Let $F:=F_{\boldsymbol{\sigma}}$.
    Given any $n$-tuple $\mathbf{v}\in(\C^d)^n$,
    we associate $n$-tuples $\mathbf{w}$ and
    $\mathbf{z}$ in $(\C^d)^n$
    given by $\mathbf{w}=A(\mathbf{v})$ and $\psi(\mathbf{z}) = \mathbf{w}$.
    Then
    \begin{eqnarray*}
        \partial_{v_k} F(\mathbf{v}) &=&
        \frac{1}{2}\left( \nabla f_{k-1}\left(\frac{v_{k-1}+v_k}{2}\right)+
        \nabla f_k \left(\frac{v_k + v_{k+1}}{2}\right)
    + i(v_{k+1}-v_{k-1}) \right)\\
                            &=& \frac{1}{2}\left(\nabla f_{k-1}(w_{k-1}) + \nabla f_k (w_k)
                            \right) + i(w_k - w_{k-1})\\
                            &=& i(z_k - \sigma_{k-1}(z_{k-1})).
    \end{eqnarray*}
    where indices are seen in $\Z/n\Z$.
    Now suppose $n$ is odd, so that $A$ is an isomorphism.
    If we denote by $\xi:=(v_2,\dotsc,v_n)$ the auxiliary variables,
    we thus have $\partial_\xi F(\mathbf{v}) = 0$ if and only if
    $z_{k+1} = \sigma_k(z_k)$ for $1\leq k\leq n-1$.
    Moreover, since $v_1 = \sum_k (-1)^{k+1} w_k$,
    if $\partial_\xi F(\mathbf{v}) = 0$ then
    \begin{equation*}
        v_1 = \sum_{k=1}^n (-1)^{k+1}\frac{z_k +\sigma_k(z_k)}{2}
    = \frac{z_1 + \sigma_n(z_n)}{2},
    \end{equation*}
    as required (since $\sigma_n(z_n) = \Phi(z_1)$ recursively).

    Finally we must show that $\partial_\xi F$ is transverse to $0$.
    This is clear in the $z$-coordinates: the matrix
    \begin{equation*}
        \ud(\partial_\xi F)(\mathbf{v})\cdot A^{-1}\cdot \ud\psi(\mathbf{z}) =
        i
        \begin{bmatrix}
            -\ud\sigma_1(z_1) & I_{2d} \\
             & -\ud\sigma_2(z_2) & I_{2d} \\
             &  & \ddots & \ddots &  \\
             &  &  & -\ud\sigma_n(z_n) & I_{2d}
        \end{bmatrix}
    \end{equation*}
    is invertible.
\end{proof}

This proposition provides a quantitative way to see how close a discrete trajectory
$(z_1,\dotsc,z_n)$ given by $(v_1,\dotsc,v_n)$ is to a discrete trajectory of the
dynamics $\sigma_n\circ\cdots\circ\sigma_1$.

If $\boldsymbol{\sigma}=(\sigma_1,\ldots,\sigma_n)$ and
$\boldsymbol{\delta}=(\delta_1,\ldots,\delta_m)$, we write
$(\boldsymbol{\sigma},\boldsymbol{\delta}) =
(\sigma_1,\ldots,\sigma_n,\delta_1,\ldots,\delta_m)$.
We have the following decomposition formula:
\begin{equation}\label{eq:composition}
    \forall v_1,\dotsc ,v_{n+m}\in\C^d,\quad
    F_{(\boldsymbol{\sigma},\boldsymbol{\delta})}(\mathbf{v}) =
    F_{(\boldsymbol{\sigma},\id)}(v_1,\dotsc,v_{n+1}) +
    F_{(\boldsymbol{\delta},\id)}(v_{n+1},\dotsc,v_{n+m},v_1).
\end{equation}

The following proposition will be of special interest for us.

\begin{prop}\label{prop:stabilization}
    Let $\boldsymbol{\sigma}$ be $m$-tuple with $m$ even and
    $\boldsymbol{\delta} := (U_1,\dotsc,U_n)$ be a $n$-tuple of unitary maps
    with $n$ odd.  Assume $U_n\circ\cdots\circ U_1 = \id$.  Then the generating
    function $F_{(\boldsymbol{\sigma},\boldsymbol{\delta})}$ is equivalent to
    $F_{(\boldsymbol{\sigma},\id)}$.  More precisely, writing
    $\mathbf{v}^1:=(v_1,\ldots,v_{n+1})$ and
    $\mathbf{v}^2:=(v_{n+2},\ldots,v_{m+n})$ there exists a $\C$-linear
    isomorphism $A:(\C^d)^{m+n}\to(\C^d)^{m+n}$ of the form
    $A(\mathbf{v}^1,\mathbf{v}^2) = (\mathbf{v}^1,\mathbf{v}^2 -
    A'(\mathbf{v}^1))$ and a non-degenerated quadratic form $Q$ of
    $(\C^d)^{n-1}$ such that
    \begin{equation*}
        F_{(\boldsymbol{\sigma},\boldsymbol{\delta})}\circ A(\mathbf{v}^1,\mathbf{v}^2) =
        F_{(\boldsymbol{\sigma},\id)}(\mathbf{v}^1)
        + Q(\mathbf{v}^2),
    \end{equation*}
    with $\ind(Q) = \ind(F_{(\boldsymbol{\delta},\id)}) = 
    \ind(F_{\boldsymbol{\delta}})$.
    In fact 
    \begin{equation*}
        Q(\mathbf{v}^2) := F_{\boldsymbol{\delta}}(0,\mathbf{v}^2).
    \end{equation*}
\end{prop}

In order to prove it, we will need the following lemma.

\begin{lem}\label{lem:quadidentity}
    Let $n\in\N$ be odd.
    Let $U_1,\dotsc,U_n\in U(\C^d)$ be unitary maps
    generated by elementary quadratic generating functions
    and such that $U_n\circ\cdots\circ U_1 = \id$.
    Let $\boldsymbol{\delta} := (U_1,\dotsc,U_n)$.
    Then, writing $\mathbf{v}':=(v_2,\ldots,v_n)$ there exists a $\C$-linear isomorphism
    $B : (\C^d)^{n+1}\to(\C^d)^{n+1}$ of the form
    $B(v_1,\mathbf{v}',v_{n+1}) = (v_1,\mathbf{v}'+B'(v_1,v_{n+1}),v_{n+1})$ and
    a non-degenerated quadratic form $Q : (\C^d)^{n-1}\to\R$
    such that
    \begin{equation*}
        \forall v_1,\dotsc v_{n+1}\in\C^d,\quad
        F_{(\boldsymbol{\delta},\id)}\circ B (v_1,\mathbf{v}',v_{n+1}) =
        Q(\mathbf{v}'),
    \end{equation*}
    with $\ind(Q) = \ind(F_{(\boldsymbol{\delta},\id)}) =
    \ind(F_{\boldsymbol{\delta}})$.
    In fact $Q(\mathbf{v}') := F_{\boldsymbol{\delta}}(0,\mathbf{v}')$.
\end{lem}

\begin{proof}
    We will show that $F:=F_{(\boldsymbol{\delta},\id)}$ is a generating function
    of the $0$-section of $(\C^d)^2$ with main variable
    $x:=(v_1,v_{n+1})$ and auxiliary variable
    $\xi := \mathbf{v}'$.
    According to Proposition~\ref{prop:gf},
    $\partial_\xi F(\mathbf{v}) = 0$ implies
    $z_{k+1} = U_k z_k$ for $1\leq k\leq n-1$
    so that $z_n = U_{n-1}\cdots U_1 z_1 = U_n^{-1}z_1$.
    Since $n+1$ is even,
    one has
    \begin{equation*}
        \im(A) =\left\{ \mathbf{w}\in(\C^d)^{n+1}\ |\  \sum_k (-1)^k
    w_k = 0 \right\}
    \end{equation*}
    so that if $\partial_\xi F(\mathbf{v}) = 0$ then
    \begin{equation*}
        \sum_{k=1}^{n+1} (-1)^k\frac{z_k + U_k z_k}{2}
        = z_{n+1} - z_1 = 0.
    \end{equation*}
    Thus, for $\mathbf{v}\in(\C^d)^{n+1}$ such that $\partial_\xi
    F(\mathbf{v})=0$,
    according to Proposition~\ref{prop:gf},
    \begin{equation*}
        \partial_x F(v) =
        \begin{pmatrix}
            i(z_1 - z_{n+1}) \\
            i(z_{n+1} - U_n z_n)
        \end{pmatrix}
        = 0,
    \end{equation*}
    since we have seen that $U_n z_n = z_1$.
    We see that $\partial_\xi F$ is transverse to $0$
    easily in the $z$-coordinates.

    The lemma is now a direct application of Lemma~\ref{lem:quad0}
    In our case, it gives
    \begin{equation*}
        F_{(\boldsymbol{\delta},\id)}(v_1,\mathbf{v}'+B'(v_1,v_{n+1}),v_{n+1}) = Q(\mathbf{v}')
    \end{equation*}
    where
    \begin{equation*}
        Q(\mathbf{v}') = F_{(\boldsymbol{\delta},\id)}(0,\mathbf{v}',0)
        = F_{\boldsymbol{\delta}} (0,\mathbf{v}').
    \end{equation*}
    Since the map $B(\mathbf{v}):=(v_1,\mathbf{v}'+B'(v_1,v_{n+1}),v_{n+1})$ is
    a linear isomorphism,
    $\ind(Q) = \ind (F_{(\boldsymbol{\delta},\id)})$.
    In fact $\ind (Q) = \ind(F_{\boldsymbol{\delta}})$, 
    as can be seen by applying Lemma~\ref{lem:quad0}, this time
    to the quadratic generating function of the 0-section
    $F_{\boldsymbol{\delta}}$ with $v_1$ as main variable.
\end{proof}

\begin{proof}[Proof of Proposition~\ref{prop:stabilization}]
    This is a direct application of Lemma~\ref{lem:quadidentity}
    to the function $F_{(\boldsymbol{\delta},\id)}$
    together with the decomposition formula (\ref{eq:composition}).
\end{proof}

The following lemma will be useful to relate critical points
of Hamiltonian diffeomorphisms with a common factor.

\begin{lem}\label{lem:commonfactor}
    Let $\boldsymbol{\sigma}$, $\boldsymbol{\delta}$ and
    $\boldsymbol{\delta}'$ be respectively an $m$-tuple
    and two $n$-tuples of small Hamiltonians diffeomorphism of $\C^d$
    as above. Let $\psi$ and $\psi'$ be the diffeomorphisms $\mathbf{z}\mapsto \mathbf{w}$
    of $(\C^d)^{m+n+1}$ defined by (\ref{eq:egf}) for
    the tuples $(\boldsymbol{\sigma},\boldsymbol{\delta},\id)$ and
    $(\boldsymbol{\sigma},\boldsymbol{\delta}',\id)$ respectively.
    Let $A:\mathbf{v}\mapsto\mathbf{w}$ be the linear map of $(\C^d)^{m+n+1}$
    defined as above.
    Then for all $\mathbf{z}^1\in(\C^d)^m$, $\mathbf{z}^2,\mathbf{z}^3\in(\C^d)^n$
    and $z_{m+n+1} \in\C^d$,
    we have $(v_1,\ldots,v_m)=(v_1',\ldots,v_m')$ where
    $\psi(\mathbf{z}^1,\mathbf{z}^2,z_{n+m+1}) = A(\mathbf{v})$
    and $\psi'(\mathbf{z}^1,\mathbf{z}^3,z_{n+m+1}) = A(\mathbf{v}')$.
\end{lem}

\begin{proof}
    Under the above hypothesis,
    \begin{equation*}
        \frac{v_{m+n+1}+v_1}{2} = z_{m+n+1} = \frac{v'_{m+n+1}+v'_1}{2}
        \ \text{ and } \
        \frac{v_k + v_{k+1}}{2} = \frac{z^1_k+\sigma_k(z^1_k)}{2} =
        \frac{v'_k + v'_{k+1}}{2},
    \end{equation*}
    where $k\in\{ 1,\ldots, n\}$. So that, with matrices,
    \begin{equation*}
        A\mathbf{v}' =
        \begin{bmatrix}
            I_{2md} \\
             & *  \\
             & & I_{2d}
        \end{bmatrix}
        A\mathbf{v}.
    \end{equation*}
    The conclusion then follows from a direct computation.
\end{proof}

\subsection{Generating family of the $S^1$-action}
\label{se:S1action}

In this section we follow Théret \cite{The98} and
study generating families of the unitary (Hamiltonian)
flow $g:t\mapsto e^{-2i\pi t}$ of $\C^{d+1}$
For $|t|<1/2$, the Hamiltonian diffeomorphism
$g_t(z):= e^{-2i\pi t}z$ admits the elementary quadratic generating function
\begin{equation*}
    q_t(w) := -\tan (\pi t)\| w\|^2, \quad
    \forall w\in\C^{d+1}.
\end{equation*}
Let $\boldsymbol{\delta}_t$ be the $m$-tuple $(g_{t/(m-1)},\ldots,g_{t/(m-1)},\id)$ with
$m\geq 5$ odd such that $F_{\boldsymbol{\delta}_t}$ generates $g_t$ for
$t\in(-\varepsilon,1+\varepsilon)$ where $\varepsilon >0$ is arbitrarily fixed
(we have only put a final ``$\id$'' in $\boldsymbol{\delta}_t$
in order for us to apply Lemma~\ref{lem:commonfactor}
in a further section without trouble).

\begin{lem}\label{lem:indQ}
    With the above notations,
    \begin{equation*}
        \ind (F_{\boldsymbol{\delta}_1}) - \ind (F_{\boldsymbol{\delta}_0})
        = 2(d+1).
    \end{equation*}
\end{lem}

\begin{proof}
    According to Proposition~\ref{prop:maspath}
    (\ref{it:sum}) and (\ref{it:rot}),
    $\mas(z,(g_t)) = 2(d+1)$ for
    all $z\in\C^{d+1}$.
    Since $F_{\boldsymbol{\delta}_t}$ is a generating function of $g_t$,
    the result follows by definition of the Maslov index.
\end{proof}

\begin{lem}[{compare with \cite[Lemma~4.4]{The98}}]\label{lem:leq}
    Let $\boldsymbol{\sigma}$ be a $m'$-tuple, with $m'$ even, such that
    $F_t := F_{(\boldsymbol{\sigma},\boldsymbol{\delta}_t)} : (\C^d)^{m'+m} \to
    \R$ is a smooth family of conical generating functions.  Then
    \begin{enumerate}[(i)]
        \item \label{item:leq}
            $\partial_t F_t(\mathbf{v}) \leq 0$,
            $\forall \mathbf{v}\in (\C^{d+1})^{m'+m}$,
        \item \label{item:sigma}
            $\partial_t F_t(\mathbf{v}) < 0$,
            $\forall \mathbf{v}\in \Sigma_{F_t}\setminus 0$.
    \end{enumerate}
\end{lem}

\begin{proof}
    The first property is a direct consequence of the definitions
    and the fact that $\partial_t (\tan(\pi t/m)) > 0$.
    Let $\mathbf{v}=(v_1,\dotsc,v_{m+m'})\in(\C^{d+1})^{m+m'}$
    be such that $\partial_t F_t(\mathbf{v}) = 0$.
    Then, for $m'+1\leq k < m'+m$,
    $w_k := \frac{v_k + v_{k+1}}{2} =0$
    thus $z_k = 0$ where the family $\mathbf{z}=(z_k)$
    is associated to the family $\mathbf{w}=(w_k)$ \emph{via}
    (\ref{eq:egf}) as usual.
    Thus if $\mathbf{v}\in\Sigma_{F_t}$, $z_k$ must be $0$
    for all $k$ for the sequence $(z_1,\dotsc,z_{m'+m})$
    to be the discrete dynamics of conical diffeomorphisms,
    hence $\mathbf{w} = 0$ and $\mathbf{v} = 0$.
\end{proof}

\subsection{A discrete variational principle for $\C$-equivariant Hamiltonian
diffeomorphism}\label{se:action}

Let $(\varphi_t)$ be the Hamiltonian flow of $\CP^d$
associated to the Hamiltonian map $h:[0,1]\times \CP^d\to\R$.
Let $\tilde{h}:[0,1]\times \sphere{2d+1}\to\R$ be
the $S^1$-invariant lift of $h$ defined by
$\tilde{h}_t:=h_t\circ\pi$ where $\pi:\sphere{2d+1}\to \CP^d$
is the quotient map $\pi(z):=[z]$.
Let $H:[0,1]\times\C^{d+1}\to\R$ be the $2$-homogeneous
Hamiltonian map such that $H_t(\lambda x):=\lambda^2 \tilde{h}_t(x)$
for all $x\in\sphere{2d+1}$.
It defines a $\C$-equivariant symplectic flow $(\Phi_t)$
stabilizing the Euclidean sphere $\sphere{2d+1}$ and
such that
\begin{equation}\label{eq:Phiphi}
    \pi \circ \Phi_t|_{\sphere{2d+1}} =
        \varphi_t\circ\pi,\quad
        \forall t\in[0,1].
\end{equation}
This flow $(\Phi_t)$ is uniquely defined by the choice of Hamiltonian
map $(h_t)$ of $(\varphi_t)$.
In fact, if $(\Phi'_t)$ is a $\C$-equivariant Hamiltonian flow
stabilizing the sphere and such that (\ref{eq:Phiphi}), then
$\Phi'_t = e^{i\theta(t)}\Phi_t$ which boils down to a change
of equivalent Hamiltonian map $(h'_t)$ for $(\varphi_t)$.
We will usually write $\varphi:=\varphi_1\in\ham(\CP^d)$
and $\Phi:=\Phi_1\in\ham_\C(\C^{d+1})$.
Given a choice of Hamiltonian map $(h_t)$,
the \emph{action} $a(x)\in\R/\Z$ of a fixed point $x\in\CP^d$
is defined by
\begin{equation*}\label{eq:defaction}
    a(x) := -\frac{1}{\pi}\left(
    \int_D \omega + \int_0^1 h_t\circ\varphi_t (x)\ud t\right)
    \in \R/\Z,
\end{equation*}
where $D\subset \CP^d$ is a $2$-disc filling the contractible loop
$\gamma := (\varphi_t)_{t\in[0,1]}$, that is $\partial D = \gamma$
(the $1/\pi$ factor is a standard renormalization to simplify notations).
Fixed points $x\in\CP^d$ of action $a\in\R/\Z$ are in one-to-one
correspondence with $\C$-lines $\C X\subset \C^{d+1}$
such that $\Phi_1(X) = e^{2i\pi a}X$, $X\in\C^{d+1}\setminus 0$
(see \cite[Prop.~5.8]{The98}).
Since the action only depends on the choice of lift $(\Phi_t)$,
when such a lift is given, we will simply call it
the action of $(\varphi_t)$ or 
the action of $\varphi$.

Following Théret,
we now define a map $\action : M\to\R$ that provides a variational principle
for fixed points of $(\varphi_t)$.
Let $\varepsilon>0$, let  $(\boldsymbol{\delta}_t)$ be one of the families
of odd tuples associated to $(g_t)$ defined in Section~\ref{se:S1action} for
$t\in(-\varepsilon,1+\varepsilon)$ and let $(\boldsymbol{\sigma}_s)$ be an even
continuous family of tuples associated to $(\Phi_s)$ of the form
$(\boldsymbol{\sigma}'_s,\boldsymbol{\delta}_0)$. Then $F_{s,t} :=
F_{(\boldsymbol{\sigma}_s,\boldsymbol{\delta}_t)}:\C^{N+1}\to\R$ gives us a
family of conical functions generating $e^{-2i\pi t}\Phi_s$.
In order to simplify notation, let $F_t := F_{1,t}$ be the family of
conical functions generating $e^{-2i\pi t}\Phi$, $t\in(-\varepsilon,1+\varepsilon)$.
Let $\tilde{f}:(-\varepsilon,1+\varepsilon)\times\sphere{2N+1}\to\R$
be the $S^1$-invariant function $\tilde{f}(t,\zeta) := F_t(\zeta)$ for $|\zeta|=1$
and $f:(-\varepsilon,1+\varepsilon)\times\CP^N \to\R$
be the induced function.
Then there is a one-to-one correspondence between
fixed points of $\varphi$ of action 
$\bar{t}\in\R/\Z$
and critical points of $f(t,\cdot)$ with value $0$ for 
any $t\in (-\varepsilon,1+\varepsilon)$.

According to property (\ref{item:sigma}) of Lemma~\ref{lem:leq},
the differential $\ud\tilde{f} = \partial_t (F_t) \ud t + \ud F_t$
never vanished on $\C^{N+1}\setminus 0$
so $0$ is a regular value of $f$.
Let $I:=(-\varepsilon,1+\varepsilon)$.
Let $M := \{ (t,\zeta)\in I\times \CP^N \ |\ f(t,\zeta) = 0 \}$
and $\action: M\to I$ be the projection onto the first factor.
Fixed points of action $\bar{t}\in\R/\Z$
are in one-to-one correspondence with critical points of
$\action$ with value $t$:
more precisely $\ud\action(t,\zeta) = 0 \Leftrightarrow
\ud_\zeta f(t,\zeta) = 0$.
Moreover, if $(t,\zeta)\in M$ is a critical point of $\action$,
then the Hessian $\ud^2\action(t,\zeta)$ is equivalent as a quadratic form to
$\ud_{\zeta,\zeta}^2 f(t,\zeta)$ which is equivalent to $\ud^2
F_t(\tilde{\zeta})$ restricted to a complement of
the $\C$-line induced by $\tilde{\zeta}\in\sphere{2N+1}$,
where $\tilde{\zeta}$ is a lift of $\zeta\in\CP^N$
(because $F_t$ is conical).
Since this line $\C\tilde{\zeta}$ is included in $\ker \ud^2 F_t(\tilde{\zeta})$,
critical points $(t,\zeta)\in M$ and
$\tilde{\zeta}\in\C^{N+1}$ share the same index.
Moreover, if $z\in\CP^d$ and $Z\in\C^{d+1}$ are fixed points
associated to $\zeta\in\CP^N$ and $\tilde{\zeta}\in\C^{N+1}$
respectively, since
\begin{equation*}
    \dim\ker \ud^2 F_t(\tilde{\zeta}) =
    \dim\ker (e^{-2i\pi t}\ud\Phi(Z)-\id)
\end{equation*}
one has
\begin{equation}\label{eq:kerhess}
    \dim\ker \ud^2_{\zeta,\zeta} f(t,\zeta) = 
    \dim\ker (\ud\varphi(z)-\id) =: \nu(z).
\end{equation}

\subsection{Cohomology of sublevel sets of $\action$}\label{se:homaction}
We recall that $H^*$ denotes the singular cohomology with integral coefficients.
Let $p:I\times\CP^N\to\CP^N$ be the projection on the second space
and $i:M\hookrightarrow I\times\CP^N$ be the inclusion map.
Let $\widehat{F}_t : \CP^N \to \R$ be the $C^1$ map
induced by $F_t|_{\sphere{2N+1}}$.
According to Lemma~\ref{lem:leq}, if $s\leq t$,
then $\widehat{F}_t\leq \widehat{F}_s$ so that
$\left\{\widehat{F}_s\leq 0\right\}\subset \left\{\widehat{F}_t\leq 0\right\}$.
Thus the subspace
\begin{equation*}
    A_t := \left\{ (s,\zeta)\in (-\varepsilon,t]\times\CP^N\ |\ 
    \widehat{F}_s(\zeta)\leq 0\right\}
\end{equation*}
retracts on $t\times\left\{\widehat{F}_t\leq 0\right\}$,
hence $p$ induces an isomorphism $H^*\left(\left\{\widehat{F}_t\leq 0\right\}\right)
\to H^*(A_t)$
for all $t\in I$ and thus induces isomorphisms 
\begin{equation}\label{eq:isomp}
    p^* : H^*\left(\left\{\widehat{F}_b\leq 0\right\},
    \left\{\widehat{F}_a\leq 0\right\}\right) \to H^*(A_b,A_a),
\end{equation}
for $a\leq b$ in $I$.
Let $a\leq b$ in $I$ and $e>0$ such that $a-e\in I$,
the subspace $A_b$ retracts on $\{\action\leq b\}\cup A_{a-e}$
by $(t,\zeta)\mapsto (s,\zeta)$ where $s$ is the maximal $r\in (a-e,t]$
satisfying $\widehat{F}_r(\zeta)=0$ or $s=a-e$ if such a max does not exist.
By excision, we then have that $i$ induces an isomorphism
\begin{equation}\label{eq:isomi}
    i^* : H^*(A_b,A_a) \to H^*(\{\action\leq b\},\{\action \leq a\}),
\end{equation}
for all $a\leq b$ in $I$.
Putting (\ref{eq:isomp}) and (\ref{eq:isomi}) together, we get the following
\begin{lem}\label{lem:isocohom}
    For all $a\leq b$ in $I$, the composition $p\circ i$ induces
    an isomorphism in cohomology
    \begin{equation*}
        H^*\left(\left\{\widehat{F}_b\leq 0\right\},
        \left\{\widehat{F}_a\leq 0\right\}\right) \simeq
        H^*(\{\action\leq b\},\{\action \leq a\}).
    \end{equation*}
\end{lem}
This result naturally generalizes by replacing large inequality by strict ones
on one or both sides of each topological pairs.
It also extends to local cohomology, the precise statement
being given in the next section.

\subsection{Local cohomology of a fixed point}\label{se:lochom}
Let $z\in\CP^d$ be a fixed point of $\varphi$
with $\Phi(Z) = e^{2i\pi t}Z$ where $Z\in\sphere{2d+1}$
is a lift of $z$.
We denote by $\lochom^*(z,t)$ the local cohomology of $\action$
at the critical point $(t,\zeta)$ corresponding to $z$, \emph{i.e.}
\begin{equation*}
    \lochom^* (z,t) = H^*(\{ \action \leq t\},\{\action \leq t\}\setminus (t,\zeta)).
\end{equation*}
This group depends only on the germ of $\action$ at $(t,\zeta)$.
Namely, for all neighborhoods $U\subset M$
of $(t,\zeta)$,
\begin{equation*}
    \lochom^* (z,t) \simeq H^*(U\cap\{ \action \leq t\},
    U\cap \{\action \leq t\}\setminus (t,\zeta)).
\end{equation*}
By an argument similar to the proof of Lemma~\ref{lem:isocohom}, the map $p\circ i : M\to\CP^N$
induces an isomorphism
\begin{equation*}
    \lochom^* (z,t) \simeq H^*\left(\left\{ \widehat{F_t} \leq 0\right\},
    \left\{ \widehat{F_t}\leq 0\right\}\setminus \zeta\right).
\end{equation*}
Thus $\lochom^*(z,t)$ is isomorphic to the local cohomology
of $\widehat{F}_t$ at the point $\zeta$ which we denote
by $\lochom^*(\widehat{F}_t;\zeta)$.
The support of a cohomology group $C^*$ is defined by
\begin{equation*}
    \supp C^* := \{ k\in\Z\ |\ C^k \neq 0\} \subset \Z.
\end{equation*}
A classical result due to Gromoll-Meyer
\cite[remark following Lemma~1]{GM69b}
implies that for any smooth function $f:M\to\R$
and any isolated critical point $x\in M$,
\begin{equation*}
    \supp \lochom^*(f;x) \subset \left[ \ind(x,f),\ind(x,f)
    + \dim\ker\ud^2 f(x) \right].
\end{equation*}
According to (\ref{eq:kerhess}), we thus have
\begin{equation}\label{eq:suppind}
    \supp \lochom^*(z,t) \subset \left[ \ind(\tilde{\zeta},F_t),\ind(\tilde{\zeta},F_t)
    + \nu(z) \right],
\end{equation}
where $\tilde{\zeta}\in\C^{N+1}$ is a lift of $\zeta$.

We want to study the relationship between the local cohomology
groups $\lochom^*(z,t)$ and $\lochom^*(z,t+1)$ when $z$ is a fixed point
of action $t\in(-\varepsilon,\varepsilon)$.
Let us assume that $t=0$ and let $(j,\zeta_j)\in M$, $j=0,1$,
be the critical points associated to $z$.
Let $(\mathbf{u}^j,\mathbf{v}^j)\in (\C^{d+1})^{n+1}\times(\C^{d+1})^{m-1}$ be
lifts of $\zeta_j$.
According to Lemma~\ref{lem:commonfactor},
one can take $\mathbf{u}^0=\mathbf{u}^1$.
According to proposition~\ref{prop:stabilization},
there exist $\C$-linear maps $A_j:(\C^{d+1})^{m+n}\to(\C^{d+1})^{m-1}$
such that $F_j(\mathbf{u},A_j(\mathbf{u},\mathbf{v}))=g(\mathbf{u})+Q_j(\mathbf{v})$
where $g=F_{(\boldsymbol{\sigma},\id)}$,
$Q_j(\mathbf{v})=F_{\boldsymbol{\delta}_j}(0,\mathbf{v})$
and the linear maps have the form $A_j(\mathbf{u},\mathbf{v})=\mathbf{v}+B_j(\mathbf{u})$.
Since the $(\mathbf{u}^0,\mathbf{v}^j)$'s are critical points 
and the $Q_j$'s are non-degenerated
$A_j(\mathbf{u}^0,0)=\mathbf{v}^j$.
\begin{lem}
    Let $C:(\C^{d+1})^{n+1}\to(\C^{d+1})^{m-1}$ be the linear map
    such that the following diagram commutes:
    \begin{equation}\label{cd:stab}
        \begin{gathered}
        \xymatrixcolsep{7pc}
        \xymatrix{
            \{ g\leq 0\}\times \{ Q_1\leq 0\} \ar[r]^-{(\mathbf{u},\mathbf{v})\mapsto
            (\mathbf{u},A_1(\mathbf{u},\mathbf{v}))} &
            \{ F_1 \leq 0 \} \\
            \{ g\leq 0\}\times \{ Q_0 \leq 0\} \ar[u]^-{(\mathbf{u},\mathbf{v})\mapsto
            (\mathbf{u},\mathbf{v}+C(\mathbf{u}))}
            \ar[r]^-{(\mathbf{u},\mathbf{v})\mapsto (\mathbf{u},A_0(\mathbf{u},\mathbf{v}))} &
            \{ F_0 \leq 0\} \ar@{^{(}->}[u]
        }
    \end{gathered}
    \end{equation}
    The set $\{ Q_0\leq 0\}$ is included in $\{ Q_1\leq 0\}$
    and the left hand vertical arrow is homotopic to the inclusion map
    $\{g\leq 0\}\times\{ Q_0\leq 0\} \hookrightarrow \{g\leq 0\}\times\{ Q_1\leq 0\}$.
\end{lem}

\begin{proof}
    According to Lemma~\ref{lem:leq}, $Q_1\leq Q_0$ so that we have the
    inclusion of their sublevel sets.
    We now prove that the map $(\mathbf{u},\mathbf{v})\mapsto 
    (\mathbf{u},\mathbf{v}+sC(\mathbf{u}))$ is well
    defined from $\{ g\leq 0\}\times \{ Q_0\leq 0\}$ to
    $\{g\leq 0\}\times\{ Q_1\leq 0\}$ for $s\in[0,1]$
    in order to conclude.
    By the above inclusion, this is true for $s=0$.
    For all $(\mathbf{u},\mathbf{v})\in\{ g\leq 0\}\times\{ Q_0\leq 0\}$, we have
    $Q_1(\mathbf{v}+C(\mathbf{u}))\leq 0$ by definition of $C$
    thus for $s\in(0,1]$, since
    $(\mathbf{u},\mathbf{v}/s)$ is also in $\{ g\leq 0\}\times\{ Q_0\leq 0\}$, one has
    \begin{equation*}
        Q_1(\mathbf{v}/s + C(\mathbf{u})) = s^{-2}Q_1(\mathbf{v}+sC(\mathbf{u}))\leq 0.
    \end{equation*}
    Hence the result for $s\neq 0$.
\end{proof}

\begin{lem}\label{lem:cdhomstab}
    In the diagram (\ref{cd:stab}),
    the sublevel sets $\{ Q_j\leq 0\}$ retract on the maximal negative
    $\C$-subspaces of the $Q_j$'s through $\C$-linear maps.
    By taking the projection of these spaces on $\CP^N$,
    we get the following commutative diagram in cohomology:
    \begin{equation*}
        \xymatrix{
            H^*( X *\CP^{k+d+1}) \ar[d] & H^*\left(\left\{\widehat{F}_1\leq 0\right\}\right)
            \ar[l]^-{\simeq} \ar[d] \\
            H^*( X *\CP^k) & H^*\left(\left\{\widehat{F}_0\leq 0\right\}\right) \ar[l]^-{\simeq}
        }
    \end{equation*}
    where $X=\{ \widehat{g}\leq 0\}$,
    $\CP^k\subset\CP^{k+d+1}$ are projectively embedded subspaces of $\CP^N$
    (the projections of the negative spaces of the $Q_j$'s),
    vertical maps are induced by inclusions and
    the horizontal maps are isomorphisms induced by 
    $[\mathbf{u}:\mathbf{v}]\mapsto[\mathbf{u}:A_j(\mathbf{u},\mathbf{v})]$.
    The same is true when replacing $\{\widehat{F}_j\leq 0\}$ by
    $\{\widehat{F}_j\leq 0\}\setminus [\mathbf{u}^0:\mathbf{v}^j]$ and
    $X*\CP^{k+j(d+1)}$ by $X *\CP^{k+j(d+1)}\setminus [\mathbf{u}^0:0]$.
\end{lem}

\begin{proof}
    The variation of dimension between the maximal negative subspaces of $Q_1$
    and of $Q_0$ is $2(d+1)$ since
    \begin{equation*}
        \ind(Q_1)-\ind(Q_0) = \ind(F_{\boldsymbol{\delta}_1}) -
        \ind(F_{\boldsymbol{\delta}_0}) = 2(d+1),
    \end{equation*}
    according to Proposition~\ref{prop:stabilization} together with
    Lemma~\ref{lem:indQ}.
    We only need to prove that the horizontal maps are actually isomorphisms.
    The last statement will then go on the same lines since the horizontal arrows
    comes from maps sending the $[\mathbf{u}^0:0]$'s on the $[\mathbf{u}^0:\mathbf{v}^j]$.
    It boils down to the fact that $\left\{\widehat{g+Q_j}\leq 0\right\}$
    retracts on $\left\{\widehat{g}\leq 0\right\}*\left\{\widehat{Q_j}\leq 0\right\}$,
    which is proved by Givental in \cite[Proposition~B.1]{Giv90}.
\end{proof}

According to this lemma, regarding cohomology of sublevel sets, one can
essentially assume that $\{\widehat{F}_0\leq 0\} = X*\CP^{k+n+1}$ and
$\{\widehat{F}_1\leq 0\} = X*\CP^k$.

\begin{prop}\label{prop:lochomper}
    Let $z\in\CP^d$ be a fixed point of $\varphi$
    with action $t\in (-\varepsilon,\varepsilon)$,
    then
    \begin{equation*}
        \lochom^*(z,t+1) \simeq \lochom^{*-2(d+1)}(z,t).
    \end{equation*}
\end{prop}
\begin{proof}
    Let us first assume that $t=0$.
    Since we are in the hypothesis of Lemma~\ref{lem:cdhomstab},
    we keep the same notations with $j\in\{0;1\}$.
    The changes of coordinate $[\mathbf{u}:\mathbf{v}]\mapsto 
    [\mathbf{u}:A_j(\mathbf{u},\mathbf{v})]$ induce isomorphisms
    in local cohomologies:
    \begin{equation}\label{eq:excisioncover}
        \lochom^*(z,j) \simeq H^*
        \left(X *\CP^{k+j(d+1)}, X*\CP^{k+j(d+1)}\setminus [\mathbf{u}^0:0]\right).
    \end{equation}
    We recall that, according to the proof of Proposition~\ref{prop:join},
    $X*\CP^{k+j(d+1)}\setminus \CP^{k+j(d+1)}$ is a complex fiber bundle
    of fibers $\C^{k+1+j(d+1)}$.
    Let $U\subset X$ be an open set of
    trivialization containing $[\mathbf{u}^0]\in X$.
    Then by excision,
    \begin{equation*}
        \lochom^*(z,j) \simeq H^*\left(U\times\C^{k+1+j(d+1)}, U\times\C^{k+1+j(d+1)}\setminus
        ([\mathbf{u}^0],0)\right),
    \end{equation*}
    and Künneth formula gives
    \begin{equation}\label{eq:Klochom}
        \lochom^*(z,j) \simeq H^{*-2(k+1+j(d+1))}\left(U,U\setminus [\mathbf{u}^0]\right),
    \end{equation}
    which concludes.

    Let us assume that $t\in (-\varepsilon,\varepsilon)$.
    We recall that $\boldsymbol{\sigma}_1 = 
    (\boldsymbol{\sigma}_1',\boldsymbol{\delta}_0)$,
    thus $s\mapsto (\boldsymbol{\sigma}_1',\boldsymbol{\delta}_{st},
    \boldsymbol{\delta}_{(1-s)t})$
    is a continuous family of tuples associated to the constant
    flow $e^{-2i\pi t}\Phi$ starting at 
    $(\boldsymbol{\sigma},\boldsymbol{\delta}_t)$.
    Therefore,
    according to Lemma~\ref{lem:isotopy},
    there exists an isotopy $(B_s)$ with $B_0=\id$ and
    $F_{(\boldsymbol{\sigma'}_1,\boldsymbol{\delta}_t,\boldsymbol{\delta}_0)}\circ
    B_1 = F_t$.
    Let $G_s := F_{(\boldsymbol{\sigma'}_1,\boldsymbol{\delta}_t,\boldsymbol{\delta}_s)}$,
    then we have $G_0\circ B_1 = F_t$ and, by the same way,
    we have $G_1\circ B_1' = F_{t+1}$ for another isotopy $(B'_s)$.
    The proof goes on the same lines as before replacing $(F_s)$
    by $(G_s)$.
\end{proof}

We want to elaborate on this local statement when the subgroup
$\lochom^*(z,t)$ ``persists in the action window $(t,t+1]$''.
Let $z\in\CP^d$ be a fixed point of $\varphi$ with action
$t\in (-\varepsilon,\varepsilon)$.
By excision,
we have an isomorphism
\begin{equation}\label{eq:locexcision}
    H^*(\{\action\leq t\},\{\action < t\}) \simeq
    \bigoplus_{i} \lochom^*(z_i,t),
\end{equation}
where $z_1=z,z_2,z_3,\ldots$ is the finite family of fixed points
of $\varphi$ with action $t$.
We will make use of the following maps induced by inclusions:
\begin{equation*}
    j_1^*: H^*(\{\action\leq t+1\},\{\action <t\}) \to
    H^*(\{\action\leq t\},\{\action <t\})
\end{equation*}
and
\begin{equation*}
    j_2^*: H^*(\{\action\leq t+1\},\{\action <t+1\}) \to
    H^*(\{\action\leq t+1\},\{\action <t\}).
\end{equation*}
We recall that $u^{d+1}\in H^{2(d+1)}(\CP^N)$ acts on
these relative cohomology groups by the cup-product,
identifying $H^*(\CP^N)$ with $H^*(I\times\CP^N)$
\emph{via} the projection $I\times\CP^N\to\CP^N$.

\begin{prop}\label{prop:globallochom}
    Let $z\in\CP^d$ be a fixed point of $\varphi$ with zero action
     such that 
    there exists a subgroup $G(z)$ of $H^*(\{\action \leq 1\},
    \{\action<0\})$ whose image $j_1^*G(z)$ is $\lochom^*(z,0)$
    under the identification (\ref{eq:locexcision}).
    Then $\lochom^*(z,1)$ trivially intersects $\ker j_2^*$ and
    we have the isomorphism
    \begin{equation*}
        \lochom^*(z,1) \simeq j_2^*\lochom^*(z,1) =
        u^{d+1} G(z).
    \end{equation*}
\end{prop}

\begin{proof}
    In order to simplify notations,
    let $M^{\leq j} := \{\action\leq j\}$ and
    $M^{< j} := M^{\leq j}\setminus [\mathbf{u}^0:\mathbf{v}^j]$ where $j\in\{0;1\}$.
    The proposition is then a direct consequence of the commutation of the following
    diagram:
    \begin{equation*}
        \xymatrixcolsep{5pc}
        \xymatrix{
            H^{*-2(d+1)}(M^{\leq 1},M^{<0}) \ar[r]^-{u^{d+1}\smile\cdot}
            \ar[d] & H^*(M^{\leq 1},M^{<0}) \\
            H^{*-2(d+1)}(M^{\leq 0},M^{<0}) \ar[r]^-{\simeq} &
            H^*(M^{\leq 1},M^{<1}) \ar[u] 
        }
    \end{equation*}
    where the vertical arrows are induced by inclusion and the bottom one
    is constructed as follows.
    By definition $H^*(M^{\leq j},M^{<j})=\lochom^*(z,j)$.
    In the last proof, we have seen that excision gives
    (\ref{eq:excisioncover}) after identifying $M^{\leq j}$
    with $X *\CP^{k+j(d+1)}$ and taking a trivialization neighborhood
    $U\subset X$ of the induced covering.
    At last, we had an isomorphism (\ref{eq:Klochom}) between local cohomologies
    $\lochom^*(z,j)$ and $H^*(U,U\setminus [\mathbf{u}^0])$ with a shift in degrees.
    This last isomorphism is in fact induced by the cup-product of
    an element of $H^*(U,U\setminus [\mathbf{u}^0])$ with the restriction of
    $u^{k+d+1}$ since the Thom class of the fiber bundle covering $X$
    is also a restriction of $u^{k+d+1}$ according to Proposition~\ref{prop:join}.
    Thus, the cup-product by $u^{d+1}$ makes explicit the isomorphism
    $H^*(M^{\leq 1},M^{<1})\xrightarrow{\simeq} H^{*-2(d+1)}(M^{\leq 0},M^{<0})$
    (after applying a suitable excision),
    so that the commutativity of the diagram follows.
\end{proof}

\subsection{Iteration properties of $\action$}

Given a $n$-tuple $\boldsymbol{\sigma}$ and a integer $m\geq 0$,
we denote by $\boldsymbol{\sigma}^m$ the $mn$-tuple
$(\boldsymbol{\sigma},\ldots,\boldsymbol{\sigma})$.
For $m\geq 1$,
let $(\boldsymbol{\sigma}^{(m)}_s)$ be the family of tuples of $(\Phi_{ms})$
defined by
\begin{equation*}
    \boldsymbol{\sigma}^{(m)}_{k/m+s} = \left( \boldsymbol{\sigma}_1^k,
    \boldsymbol{\sigma}_s, \boldsymbol{\sigma}^{m-k-1}_0\right),\quad \forall s\in[0,1/m],
    0\leq k\leq m-1.
\end{equation*}
Let $F^m_{s,t} := F_{(\boldsymbol{\sigma}_s^{(m)},\boldsymbol{\delta}_t)}$ be
the induced generating family
of $e^{-2i\pi t}\Phi_{ms}$.
We denote by 
\begin{equation*}
    M_m :=
    \left\{ (t,[\zeta])\in (-\varepsilon,1 + \varepsilon)\times
    \CP^{N(m)} \ |\ F^m_{1,t}(\zeta) = 0 \right\},\quad 
    \action^m : M_m \to (-\varepsilon,1+\varepsilon)
\end{equation*}
the discrete action associated to $\Phi^m$ \emph{via}
the generating family $F^m_{1,t}:\C^{N(m)+1}\to\R$.

According to Proposition~\ref{prop:lochomper},
in order to study the local cohomology groups of $\action$
and $\action^m$, it is enough to consider points of value in $[0,1)$.
Let $y\in\CP^{d}$ be a fixed point of $\varphi$
and $t(y)\in [0,1)$ be uniquely defined by
$\Phi(\tilde{y}) = e^{2i\pi t(y)}\tilde{y}$,
where $\pi(\tilde{y})=y$.
We define the \emph{index} of the fixed point $y$
by $\ind(y) := \ind(\zeta,F_{1,t(y)})$.
We extend these definitions to the iterated diffeomorphism $\varphi^m$
the following way:
if $y\in\CP^d$ is a fixed point of $\varphi$,
$t(y^m)$ denotes the
only $t\in [0,1)$ such that
$\Phi^m(\tilde{y}) = e^{2i\pi t}\tilde{y}$,
hence satisfies
\begin{equation*}
    t(y^m) = mt(y) - \lfloor mt(y) \rfloor.
\end{equation*}
The same way, the $m$-iterated index of $y$ designates the integer
\begin{equation*}
    \ind (y^m) := \ind \left(\zeta, F^m_{1,t}\right),
\end{equation*}
for some critical point $\zeta$ associated to the fixed point $\tilde{y}$
of the diffeomorphism $e^{-2i\pi t(y^m)}\Phi^m$.

According to (\ref{eq:suppind}), 
\begin{equation*}
    \supp\lochom^*(y^m,t(y^m))\subset \left[ \ind(y^m),\ind(y^m)+\nu(y^m)\right].
\end{equation*}
By definition of the Maslov index,
\begin{equation*}
    \ind(y^m) =
    \mas\left(\tilde{y},\left(e^{-2i\pi t(y^m)s}\Phi_{ms}\right)\right) 
    + i(m),
\end{equation*}
where $i(m):=\ind(F^m_{0,0})$ only depends on $m\in\N^*$.
Thus, according to Bott iteration inequalities (\ref{eq:mmasequiv}),
\begin{equation}\label{eq:suppmind}
    \supp\lochom^*(y^m,t(y^m)) \subset
    \left[ \mind (y^m) - d, \mind (y^m) +d
    \right],
\end{equation}
where $\mind(y^m) := i(m)+ \mmas(\tilde{y},(e^{-2i\pi t(y^m)s}\Phi_{ms}))$.

\begin{lem}\label{lem:indmas}
    Let $y\in\CP^d$ be a fixed point of $\varphi$, then
    \begin{equation*}
        \mind(y^m) =
        m\mmas\left(\tilde{y},\left(e^{-2i\pi t(y)s}\Phi_{s}\right) \right)
        - 2(d+1)\lfloor mt(y) \rfloor + i(m),\quad
        \forall m\in\N^*.
    \end{equation*}
\end{lem}

\begin{proof}
    Let $y\in\CP^d$ be fixed by $\varphi$, $m\in\N^*$ and $k\in\N^*$.
    In $\textup{Sp}(2(d+1))$, the path
    $s\mapsto \ud\left(e^{-2i\pi kmt(y)s}\Phi_{kms}\right)(\tilde{y})$
    is homotopic relative to endpoints to the concatenation of
    the path $s\mapsto \ud\left(e^{-2i\pi kt(y^m)s}\Phi_{kms}\right)(\tilde{y})$
    and the loop $\Gamma :s\mapsto e^{-2i\pi k\lfloor mt(y)\rfloor s }$,
    thus Proposition~\ref{prop:maspath} (\ref{it:comp}) and (\ref{it:homot})
    implies
    \begin{equation*}
        \mas\left(\tilde{y},\left(e^{-2i\pi kmt(y)s}\Phi_{kms}\right) \right) =
        \mas\left(\tilde{y}, \left(e^{-2i\pi kt(y^m)s}\Phi_{kms}\right)\right) +
        \mas(\Gamma),
    \end{equation*}
    According to Proposition~\ref{prop:maspath},
    $\mas(\Gamma) = 2(d+1)k\lfloor mt(y)\rfloor$,
    thus, dividing by $k$ and letting $k\to\infty$, we get
    \begin{equation*}
        \mmas\left(\tilde{y},\left(e^{-2i\pi mt(y)s}\Phi_{ms}\right) \right) =
        \mmas\left(\tilde{y}, \left(e^{-2i\pi t(y^m)s}\Phi_{ms}\right)\right) +
        2(d+1)\lfloor mt(y)\rfloor.
    \end{equation*}
\end{proof}

\subsection{Remarks on the parallel with Floer homology}
\label{se:discussion}
Our construction of the cohomology groups
of the sublevel sets depends \emph{a priori} on the choice of the tuple
$\boldsymbol{\sigma}$ of the Hamiltonian diffeomorphism $\Phi$.
We discuss uniqueness properties of this group and 
links with Floer homology without complete proofs,
as it will not be necessary for this paper.
Given a $\C$-equivariant flow $(\Phi_s)$
and real numbers $a<b$ not in the action spectrum of $\varphi$,
we can define $F_{s,t}$ generating $e^{-2i\pi t}\Phi_s$
for $s\in[0,1]$ and $t\in I$ where $I$ is an interval 
containing some $a'<a$ and $b'>b$, the same way as before.
Let $i_0 := \ind(F_{0,0})$, we define the cohomology group
\begin{equation*}
    G_{(a,b)}^*((\Phi_s)) :=
    H^{*-i_0}(\{\action\leq b\},\{\action\leq a\}).
\end{equation*}
By using Lemma~\ref{lem:isotopy} together with
Lemma~\ref{lem:quadidentity}, one can show that
this group does not depend on the specific choice of $(\boldsymbol{\sigma}_s)$
associated to the Hamiltonian flow $(\Phi_s)$.
Moreover, one can show that it only depends on the homotopy class
of Hamiltonian path $(\Phi_s)$ relative to endpoints
the same way Théret showed it for its rotation numbers
in \cite[Prop.~5.7]{The98}.
We can go a little further if we normalize the Hamiltonian
map $(h_s)$ of the flow $(\varphi_s)$ so that for instance
\begin{equation*}
    \int_{\CP^d} h_s \omega^d = 0,\quad\forall s\in[0,1].
\end{equation*}
Then $(\Phi_s)$ is uniquely determined by $(\varphi_s)$
and the group $G_{(a,b)}^*((\Phi_s))$ only
depends on the homotopy class of $(\varphi_s)$ relative to endpoints,
that is the choice of a lift $\widetilde{\varphi}\in\widetilde{\ham}(\CP^d)$
of $\varphi$ to the universal cover of $\ham(\CP^d)$.

When there are finitely many fixed points,
the building blocks of $G^*_{(a,b)}$ are the local cohomology
groups $\lochom^{*-i_0}(z,t)$ for fixed points $z\in\CP^d$ of
$\varphi$ and $t\in (a,b)$.
Giving a couple $(z,t)$ is equivalent to giving a capped orbit
$\bar{z}=(z,u_t)$ where $u_t$ is the capping naturally induced by
the lift $e^{-2i\pi t}\Phi$, as seen in Section~\ref{se:masequiv}.
We can remark that both $\lochom^*(\bar{z}):=\lochom^{*-i_0}(z,t)$
and local Floer cohomology $HF^*(\bar{z})$ have their
support in $[\mmas(\bar{z})-d,\mmas(\bar{z})+d]$
and that they are equal when $\bar{z}$ is non-degenerate.
One can prove that they are isomorphic by using
the isomorphism between cohomology of generating functions
of Hamiltonian diffeomorphisms of $\C^{d+1}$ and their
Floer cohomology in \cite{Vit96}.

Proposition~\ref{prop:globallochom} essentially asserts that
for all capped orbit $\bar{z}$ of $\varphi$,
\begin{equation*}
    \lochom^*(A_0\# \bar{z}) = 
    u^{d+1} \lochom^*(\bar{z}),
\end{equation*}
where $A_0$ is the generator of
$\pi_2(\CP^d)\simeq\Z$ satisfying $\la[\omega],A_0\ra = -\pi$.
The reader familiar to quantum homology can interpret this relation
as the fact that 
the class $u^{d+1}$ of $H^*(\CP^\infty)$ acts on the group $G^*_{(a,b)}$
the same way as
the operator $\textup{q}$ of the quantum homology of $\CP^d$ acts on
the group $HF^{(a,b)}_*$
(see \cite[Sect.~2]{GG18} but beware signs of $\omega$ and action
are opposite to our convention).
Moreover, the relation
\begin{equation*}
    [\textup{pt}] * [\CP^{d-1}] = \textup{q}[\CP^d]
\end{equation*}
in the quantum homology of $\CP^d$ can be interpreted as
\begin{equation*}
    u^d\smile u = u^{d+1}\smile 1,
\end{equation*}
seeing $u^d$, $u$ and $1$ respectively as the Poincaré duals of the homology classes
$[\textup{pt}]$, $[\CP^{d-1}]$ and $[\CP^d]$ in $H_*(\CP^d)$.
This relation is fundamental in the original proof of Ginzburg-Gürel
and thus explains the fundamental role of the subordinated classes $1$, $u$
and $u^{d+1}$ which will be played in our proof.

\section{Proof of Theorem~\ref{thm:main} and its corollaries}\label{se:proof}

In this section, we prove Theorem~\ref{thm:main},
postponing the proof of the Crossing energy theorem
to Section~\ref{se:crossing}.
We then provide the proofs of Corollaries~\ref{cor:hyperbolic}
and \ref{cor:pseudorotation} sketched in the introduction.

\subsection{Preliminaries}

Let $\varphi\in\ham(\CP^d)$ be the time-one map
of a Hamiltonian flow $(\varphi_s)$
with a fixed point $x\in\CP^d$ which is isolated as
an invariant set.
Moreover, let $(\Phi_s)$ be a $\C$-equivariant Hamiltonian
flow lifting $(\varphi_s)$ and let us suppose that
$t(x)=0$ and that
the local cohomology groups
of $x$ associated to the iterations of $\varphi$ are all non-zero.
We will prove Theorem~\ref{thm:main} by contradiction:
let us assume that $\varphi$ has only finitely many
periodic points
so that $\action^m$ has only isolated critical points
in a finite number for all $m\in\N^*$.
In our construction of $\action$, we take
$\varepsilon< 1/2$ so that any fixed point of $\varphi$
have at most $2$ associated critical points.
Taking an iteration, we might suppose that any periodic
point is a fixed point of $\varphi$.
For all $m\in\N^*$, let $(j,\zeta^m_j)\in M_m$,
$j\in\{ 0,1\}$, be the critical points of $\action$
associated to $x$.
In Section~\ref{se:crossing}, we prove the Crossing energy theorem
which applies to our point $x$, isolated as an invariant set, 
in the following way:

\begin{thm}\label{thm:aux}
    There exist $c_\infty >0$,
    families of open neighborhoods $V^m_j,W^m_j\subset M_m$
    of $(j,\zeta^m_j)$ with $W^m_j\subset V^m_j$ which do not intersect 
    $\textup{Crit}(\action^m)\setminus (j,\zeta^m_j)$
    and an adapted pseudo-gradient $X_m$ of $\action^m$,
    such that any gradient or reversed-gradient flow line $\mathbf{u}:\R\to M_m$,
    $\dot{\mathbf{u}} =\pm X_m(\mathbf{u})$, with $\mathbf{u}(s)\not\in V^m_j$
    and $\mathbf{u}(t)\in W^m_j$ for some $m\in\N^*$
    and $j\in\{0,1\}$ satisfies 
    \begin{equation*}
        |\action^m(\mathbf{u}(s))-\action^m(\mathbf{u}(t))|>c_\infty.
    \end{equation*}
\end{thm}
   
Let $c_\infty > 0$ be given by the above result.
Without loss of generality, we suppose that
\begin{equation}\label{eq:cinf}
    0 < c_\infty < \frac{1}{2(d+1)}.
\end{equation}

\subsection{Augmented action}

By analogy with Ginzburg-Gürel augmented action,
for any fixed point $y\in\CP^d$ we define
\begin{equation*}
    \tilde{a}(y^m) := mt(y) - \frac{1}{2(d+1)}\mmas(\tilde{y},(e^{-2i\pi mst(y)}\Phi_{ms}))
    = m\tilde{a}(y).
\end{equation*}
According to Lemma~\ref{lem:indmas},
\begin{equation}\label{eq:amind}
    \begin{split}
        \mind(y^m)-\mind(x^m) &=
        m\left(\mmas\left(\tilde{y},\left(e^{-2i\pi t(y) s}\Phi_s\right)\right) -
        \mmas\left(\tilde{x},\left(e^{-2i\pi t(x) s}\Phi_s\right)\right) \right)\\
        &- 2(d+1)\lfloor mt(y)\rfloor\\
        &= 2m(d+1)\left(\tilde{a}(x) - \tilde{a}(y)\right) + 2(d+1)t(y^m).
    \end{split}
\end{equation}
By Dirichlet's lemma, one can find $m\in\N^*$ such that,
for all fixed point $y$,
the fraction part of each $mt(y)$, which is $t(y^m)$,
satisfies 
\begin{equation*}\label{eq:typ}
t(y^m)\in[0,c_\infty)\cup (1-c_\infty,1)
\end{equation*}
with $m$ taken sufficiently large so that
\begin{equation*}
    |\tilde{a}(y)-\tilde{a}(x)| = 0 \quad \text{ or }\quad
    m|\tilde{a}(y)-\tilde{a}(x)| > 3.
\end{equation*}
Thus equation (\ref{eq:amind}) together with assumption
(\ref{eq:cinf}) implies the following lemma.

\begin{lem}\label{lem:xclose}
    With this specific choice of $m\in\N^*$,
    given any fixed point $y\in\CP^d$, we have:
    \begin{itemize}
        \item 
        $\left|\mind(y^m)-\mind(x^m)\right| \leq 2d +1$ implies 
        $t(y^m)<c_\infty$,
    \item
        $\left|\mind(y^m)-\left(\mind(x^m)+2(d+1)\right)\right|
        \leq 2d+ 1$ implies 
        $t(y^m)>1 - c_\infty$.
    \end{itemize}
\end{lem}

Given two subsets $A,B\subset \R$, we denote the smallest distance among their
points by
\begin{equation*}
    \dist(A,B) := \inf \left\{ |a-b|\ |\ a\in A,b\in B\right\} \in [0,+\infty].
\end{equation*}
According to Proposition~\ref{prop:lochomper} and equation
(\ref{eq:suppmind}), Lemma~\ref{lem:xclose} implies

\begin{cor}\label{cor:supp}
    With this specific choice of $m\in\N^*$,
    given any fixed point $y\in\CP^d$ of $\varphi^m$ of
    action $\bar{t}$ with $t\in (-\varepsilon,1+\varepsilon)$,
    if
    \begin{equation*}
        \dist(\supp\lochom^*(y,t),\supp\lochom^*(x,j)) \leq 1,
    \end{equation*}
    then $|t-j|<c_\infty$, for $j\in\{0,1\}$.
\end{cor}

\subsection{Subordinated min-max}

By taking the $m$-th iteration of $\varphi\in\CP^d$,
we can suppose that $\varphi$ satisfies Corollary~\ref{cor:supp}
for $m=1$.
Let $G_t:M\to M$ be the gradient flow associated to the pseudo-gradient
of Theorem~\ref{thm:aux} at time $t\in\R$.
In order to simplify notation in this section,
given any subset $U\subset M$ and any $b\in I$,
we set $U^{\leq b}:= U\cap\{\action \leq b\}$ and
$U^{<b}:= U\cap\{\action < b\}$, whereas we denote by
$\lochom^*(y)$ the local homology associated to a critical points
$y\in M$.
For all critical points $y\in M$, let us define
a specific flow-out $U(y)$, that is an open neighborhood of
$y$ which is invariant under $G_t$ for all $t\geq 0$.
We take a small neighborhood $B$ of $y$, then we set
\begin{equation*}
    U'(y) := \bigcup_{t\geq 0} G_t(B),
\end{equation*}
let $\{z_j\}$ be the family of critical points in the closure of $U'(y)$,
we then define $U(y):=U'(y)\cup\{ z_j\}$ which is in fact open
if one takes $B$ small enough.
Let $z_0:=(0,\zeta_0)$ and $z_1:=(0,\zeta_1)$ be the two critical points
associated to the fixed point $x\in\CP^d$.
Applying Theorem~\ref{thm:aux} together with Corollary~\ref{cor:supp},
we choose $B$ small enough such that $z_j\in U(y)$ implies that
\begin{equation}\label{eq:supp}
    \dist(\supp\lochom^*(y),\supp\lochom^*(z_j)) > 1,
\end{equation}
and in the case where $y=z_j$, we do the same so that this last
equation holds also for critical points $y\in U(z_j)$.
We also ask that if some critical point $y$ is in $U(y')$,
then $U(y)\subset U(y')$ (this could be achieved by induction on
the critical values, starting by defining the flow-out for the critical
points of largest value).
We first prove that the local cohomology $\lochom^*(z_0)$
``persists in the action window $[0,1+\varepsilon)$''.
Let $v_0\in\lochom^*(z_0)$ be a non-zero class,
which exists by hypothesis.

\begin{lem}\label{lem:persistence}
    For all $b\in [0,1+\varepsilon)$, there exists a class
    $v\in H^*(M^{\leq b},M^{<0})$ such that its image under
    the morphism induced by the inclusion $H^*(M^{\leq b},M^{<0})
    \to H^*(M^{\leq 0},M^{<0})$ is $v_0$.
    Moreover, given one of the above flow-outs $U=U(y)$,
    \begin{equation*}
        v\not\in\ker \left(H^*(M^{\leq b},M^{<0})\to H^*(U^{\leq b},U^{<0})\right)
    \end{equation*}
    if and only if $z_0\in U$, where the morphism is induced by inclusion.
\end{lem}

\begin{proof}
    According to Morse deformation lemma, if the lemma
    is true for $b$ and $(b,c]\subset I$ does not contain
    any critical value, then the lemma is also true for $c$.
    Since there is a finite number of critical values, we can
    thus prove this lemma inductively on the critical value $b\geq 0$.
    We start with the case $b=0$. As we have seen, by excision
    $\lochom^*(z_0)\subset H^*(M^{\leq 0},M^{<0})$ and taking
    $v=v_0$ under this injection is enough.

    Let us assume that $b>0$ is a critical value and that
    the lemma is true on $[0,b)$.
    Let $(y_k)$ be the family of critical points of value $b$
    and $U_k := U(y_k)$ be their associated flows-out.
    We will work with the following commutative diagram:
    \begin{equation}\label{cd:persistence}
        \begin{gathered}
        \xymatrix{
            H^*(M^{\leq b},M^{<b}) \ar[r]^-{j^*} \ar[d] &
            H^*(M^{\leq b},M^{<0}) \ar[r]^-{i^*} \ar[d] &
            H^*(M^{< b},M^{<0}) \ar[d] \\
            H^*(U_k^{\leq b},U_k^{<b}) \ar[r]^-{j_k^*} &
            H^*(U_k^{\leq b},U_k^{<0}) \ar[r]^-{i_k^*} &
            H^*(U_k^{< b},U_k^{<0})
        }
    \end{gathered}
    \end{equation}
    where every arrow is induced by inclusion.
    By Morse deformation lemma the $M^{<b}$ and $U^{<b}$ in
    the right hand side of the diagram can be replaced by
    $M^{\leq c}$ and $U^{\leq c}$ for some $c<b$ close enough.
    By induction, there thus exists $v'\in H^*(M^{<b},M^{<0})$
    satisfying the lemma (with symbol $\leq b$ replaced by $<b$).
    Let us first show that $v'$ is in the image of $i^*$.
    According to the long exact sequence of the triple $(M^{\leq b},M^{<b},M^{<0})$,
    it boils down to showing that $\partial^*v'=0$ where
    $\partial^*$ is the coboundary map.
    By contradiction let us assume that $\partial^*v'\neq 0$.
    By excision, we recall that
    \begin{equation*}
        H^*(M^{\leq b},M^{<b}) \simeq \bigoplus_k \lochom^*(y_k),
    \end{equation*}
    thus if $\partial^*v' \neq 0$ then $\partial_k^* v'\neq 0$ for some $k$,
    where $\partial^*_k$ is the composition of the coboundary with the projection
    on $\lochom^*(y_k)$ with respect to the above direct sum.
    Identifying $H^*(U_k^{\leq b},U_k^{<b})$ with
    $\lochom^*(y_k)$ by excision, one has the following commutative diagram:
    \begin{equation*}
        \xymatrix{
            H^*(M^{< b},M^{<0}) \ar[r]^-{\partial^*} \ar[d] \ar[rd]^-{\partial^*_k} &
            H^{*+1}(M^{\leq b},M^{<b}) \ar[d] \\
            H^*(U_k^{< b},U_k^{<0}) \ar[r] &
            \lochom^{*+1}(y_k)
        }
    \end{equation*}
    where the vertical arrows are induced by inclusions and the horizontal
    are coboundary maps. Thus we see that $v'$ is not in the kernel of
    the left hand side arrow, so that by induction hypothesis $z_0\in U_k$.
    But according to equation (\ref{eq:supp}), if $\ell$ is the degree of $v'$
    (which maps to $v_0\in \lochom^\ell(z_0)$),
    then $\lochom^{\ell+1}(y_k) = 0$, a contradiction.

    Hence $\partial^* v' = 0$ and there exists $v''\in H^*(M^{\leq b},M^{<0})$
    such that $i^*v'' = v'$. This $v''$ maps to $v_0$ as required but does not
    satisfy the second conclusion of the lemma \emph{a priori}.
    We now explain how to build $v$ in the inverse image of $v'$.
    For a fixed $k$, let $v''_k\in H^*(U_k^{\leq b},U_k^{<0})$ be the image of
    $v''$ under the vertical arrow of (\ref{cd:persistence}).
    For $v''$ to satisfy the conclusions of the lemma, we need $v''_k$
    to be zero if and only if the image of $v'=i^* v''$ under its vertical
    arrow is zero. If $i^*_k v''_k = 0$, then there exists $w'_k\in H^*(U_k^{\leq b},U_k^{<b})$
    such that $j^*_k w'_k = v''_k$.
    We recall that the left hand side arrow is equivalent to the projection
    \begin{equation*}
        \bigoplus_{\ell} \lochom^*(y_\ell) \to \lochom^*(y_k),
    \end{equation*}
    let $w_k\in H^*(M^{\leq b},M^{<b})$ be then the image of $w'_k$
    under the inclusion $\lochom^*(y_k)\subset H^*(M^{\leq b},M^{<b})$.
    We finally set
    \begin{equation*}
        v := v'' - \sum_k j^*(w_k) \in H^*(M^{\leq b},M^{<0})
    \end{equation*}
    to be the wanted solution.

    The conclusion is true for the $U=U_k$
    with this choice of $v$, by construction.
    Let $U$ be the flow-out of some critical point.
    If $U$ does not contain any of the $y_k$, by the Morse deformation
    lemma, $U^{\leq b}$ retracts on $U^{<b}$ so that the conclusion follows
    by induction.
    Otherwise, let $(y_{k_q})$ be the sub-family of $(y_k)$ included in
    $U$, so that $U_{k_q}\subset U$ by construction of our flows-out.
    If $z_0\in U$, then by hypothesis, $v'$ is not in the kernel of
    $H^*(M^{<b},M^{<0})\to H^*(U^{<b},U^{<0})$ and, as $i^* v=v'$,
    $v$ is neither in the kernel of
    $H^*(M^{\leq b},M^{<0})\to H^*(U^{\leq b},U^{<0})$.
    Conversely, if $v$ is not in the above kernel, either
    $v'$ is not in the kernel of its restriction to $U$, in
    which case $z_0\in U$ by induction, or its image
    $v_U \in H^*(U^{\leq b},U^{<0})$ under
    the above map has the form $j^*_U w_U$ where
    $j^*_U : H^*(U^{\leq b},U^{<b}) \to H^*(U^{\leq b},U^{<0})$.
    Now by excision $H^*(U^{\leq b},U^{<b})$ is isomorphic
    to the direct sum of the $H^*(U_{k_q}^{\leq b},U_{k_q}^{<b})$'s.
    Thus, there is a $k_q$ such that $w_U$ projects on $w_{k_q}\neq 0$.
    The commutativity of the left hand square of (\ref{cd:persistence}) for $k=k_q$
    together with the construction of $v$ bring a contradiction.
\end{proof}

\begin{cor}\label{cor:persistence}
    There exists a subgroup $G\subset H^*(M^{\leq 1},M^{<0})$
    which image under the map $H^*(M^{\leq 1},M^{<0}) \to
    H^*(M^{\leq 0},M^{<0})$ is $\lochom^*(z_0)$ and such that
    its image under $H^*(M^{\leq 1},M^{<0})\to H^*(U^{\leq 1},U^{<0})$
    is non-zero if and only if $z_0\in U$,
    where $U:=U(y)$ is a flow-out.
\end{cor}

\begin{proof}
    Take $G$ to be the subgroup of $H^*(M^{\leq 1},M^{<0})$ generated
    by every $v$ given by Lemma~\ref{lem:persistence} for $b=1$
    and each $v_0\in\lochom^*(z_0)\setminus 0$.
\end{proof}

By using (\ref{eq:supp}) for $z_1$, one can prove dually the following
lemma.

\begin{lem}\label{lem:persistence2}
    The subgroup $\lochom^*(z_1)\subset H^*(M^{\leq 1},M^{<1})$
    trivially intersects the kernel of the map
    \begin{equation*}
        H^*(M^{\leq 1},M^{<1}) \to H^*(U^{\leq 1},U^{<0})
    \end{equation*}
    induced by inclusion, where $U:=U(z_1)$ is the flow-out of $z_1$.
\end{lem}

\begin{proof}[Proof of Theorem~\ref{thm:main}]
    Let $v\in H^*(M^{\leq 1},M^{<0})$ be the class given by
    Lemma~\ref{lem:persistence}
    for $b=1$.
    Thus, applying Proposition~\ref{prop:globallochom} to Corollary~\ref{cor:persistence},
    there exists a class $w\in\lochom^*(z_1)$ that maps to
    $u^{d+1}v\in H^*(M^{\leq 1},M^{<0})$.
    Considering the restriction map $H^*(M^{\leq 1},M^{<0})\to
    H^*(U^{\leq 1},U^{<0})$, where $U:= U(z_1)$ is the flow-out of $z_1$,
    the image of $u^{d+1}v$ is non-zero by Lemma~\ref{lem:persistence2},
    thus the image $v'$ of $v$ is non-zero.
    Therefore, Lemma~\ref{lem:persistence} implies that $z_0\in U(z_1)$.
    For $t\in [0,1]$, let $i^*_t$ be the map induced by inclusion
    \begin{equation*}
        i^*_t : H^*(U^{\leq 1},U^{<0}) \to H^*(U^{\leq t},U^{<0}),
    \end{equation*}
    and, for $0\leq k\leq d+1$, let $\tau_k\in [0,1]$ be the family of
    min-max values
    \begin{equation}\label{eq:minmax}
        \tau_k := \inf\left\{ t \geq 0 \ | \ 
        u^k v' \not\in \ker i^*_t \right\}.
    \end{equation}
    As we have just seen, $u^{d+1}v'\neq 0$ so that $u^k v'\neq 0$
    for $k\leq d+1$.
    Lemma~\ref{lem:persistence} implies that $\tau_0 = 0$.
    By the long exact sequence of the triple $(U^{\leq 1},U^{<1},U^{<0})$,
    the image of $u^{d+1}v'$ under $H^*(U^{\leq 1},U^{<0}) \to
    H^*(U^{<1},U^{<0})$ is zero, thus $\tau_{d+1} = 1$.
    By Lyusternik-Schnirelmann theory, $\tau_1$ is a critical value
    of $\action|_U$ and $\tau_0<\tau_1<\tau_{d+1}$ since $\action$ has a finite
    number of critical values (we recall that $d\geq 1$).
    Let $(y_j)$ be the family of critical points of value $\tau_1$ in $U$.
    According to Theorem~\ref{thm:aux}, if the flow-out $U$ has been taken small
    enough, $\tau_1 \leq 1-c_\infty $, thus $\tau_1 < c_\infty$ since there are no
    critical points with value in $[c_\infty,1-c_\infty]$.
    Since $H^*(U^{\leq \tau_1},U^{<\tau_1})$ decomposes in the direct sum of
    the local cohomologies of the $y_j$'s, we find by similar arguments as before
    that there exists a $j$ such that the image of $uv'$ on
    $H^*(U_j^{\leq \tau_1},U_j^{<0})$ is non-zero.
    For this $j$, the image of $v'$ under the same map is thus also non-zero,
    hence $z_0\in U_j$ by Lemma~\ref{lem:persistence}.
    But according to Theorem~\ref{thm:aux}, for $U_j$ taken small
    enough in our proofs, one must have $\tau_1 \geq c_\infty$,
    a contradiction.
\end{proof}

\subsection{Corollaries}
\begin{proof}[Proof of Corollary~\ref{cor:hyperbolic}]
    Let $x\in\CP^d$ be a hyperbolic fixed point of $\varphi\in\ham(\CP^d)$
    and $\Phi\in\ham_\C(\C^{d+1})$ be a lift of $\varphi$.
    According to Theorem~\ref{thm:main}, it is enough to prove that
    the local cohomology group $\lochom^*(x,t(x^k))$
    is non zero for all iteration $k\in\N^*$.
    In Section~\ref{se:lochom}, we have seen that
    $\lochom^*(x,t(x^k))\simeq \lochom^*(\widehat{F}_{t(x^k)};\zeta_k)$
    where $\zeta_k\in\CP^{N(k)}$ is the critical point of
    the map $\widehat{F}_{t(x^k)}:\CP^{N(k)} \to\R$ induced
    by the generating function $F_{t(x^k)}$ of
    $e^{-2i\pi t(x^k)}\Phi\in\ham_\C(\C^{d+1})$.
    Since $x$ is hyperbolic, $\dim\ker(\ud\varphi(x)^k-\id)=0$
    for all $k\in\N^*$, thus
    $\ud^2\widehat{F}_{t(x^k)}$ is non-degenerated
    according to (\ref{eq:kerhess})
    and $\lochom^*(x,t(x^k))$ has rank $1$.
\end{proof}

\begin{proof}[Proof of Corollary~\ref{cor:pseudorotation}]
    Let $\varphi\in\ham(\CP^d)$ be a pseudo-rotation
    of $\CP^d$ with fixed points $x_1,\dotsc,x_{d+1}\in\CP^d$
    and a lift $\Phi\in\ham_\C(\C^{d+1})$.
    According to Theorem~\ref{thm:main},
    it is enough to prove that the local cohomology groups
    $\lochom^*(x_j,t(x^k_j))$ are non-zero for all
    $j$ and all $k\in\N^*$.
    This is a consequence of the following fact
    due to Théret \cite{The98}:
    given any $\varphi\in\ham(\CP^d)$,
    there always exists some integer $k\in\N$
    such that the classes 
    $u^k,u^{k+1},\dotsc,u^{k+d}\in H^*(\CP^N)$ are not
    in the kernel of the map induced by the projection
    $I\times\CP^N\to\CP^N$
    \begin{equation*}H^*(\CP^N) \to \HH\end{equation*}
    with associated min-max values in $[0,1)$
    (that is the values $\tau_k$ of (\ref{eq:minmax})
    where the $u^k v'$'s are replaced by the $u^{k+j}$'s).
    We briefly recall the proof in our setting.

    Given $t\in [0,1]$, let $\ell(t)\in\N$
    be the integer
    \begin{equation*}
        \ell(t) := \max\left\{ k\in\N\ |\
            u^k\not\in\ker\left(
                H^*(\CP^N) \to 
                H^*(\{\action\leq t\})
        \right)\right\}.
    \end{equation*}
    We can show that $\ell(t+1)=\ell(t)+d+1$ by combining
    Corollary~\ref{cor:homlength} and Lemma~\ref{lem:cdhomstab},
    hence the statement for $k:=\ell(0)+1$.

    In the case where $\varphi\in\ham(\CP^d)$
    has isolated fixed points, the min-max values
    of the $u^{k+j}$'s must be different by Lyusternik-Schnirelmann
    theory.
    Thus any $\varphi\in\ham(\CP^d)$ with isolated fixed points
    has at least $d+1$ fixed points with non-zero 
    associated local cohomology.
    Hence the conclusion holds, since every iteration of a pseudo-rotation
    have only $d+1$ fixed points.
\end{proof}

\section{Ginzburg-Gürel Crossing Theorem for generating functions}
\label{se:crossing}

In this section, we prove the analogue of Ginzburg-Gürel
Crossing theorem for generating functions.
Since the proof in $\CP^d$ is essentially the
same as the one in $\C^d$ with some technical changes
which could make it less transparent to the reader,
we first provide the argument for $\C^d$, even though
the $\C^d$ setting will not be employed in this paper.

\subsection{Crossing energy theorem in $\C^d$}

If $\boldsymbol{\sigma}:=(\sigma,\dotsc,\sigma)$ is an tuple of even size associated to
$\Phi$, then $\boldsymbol{\sigma}^m$ is
a tuple of even size of the iterated diffeomorphism $\Phi^m$.
Given any $x\in\C^d$, let $B^{2d}_r(x) := \{ z\in\C^d\ |\ |z-x|< r\}$
or simply $B_r(x)$.
We will denote by $A_m$ the linear isomorphism
of $(\C^d)^{mn+1}$ defined by 
$A_m(\mathbf{v}) := \mathbf{w}$ where $w_k =\frac{v_k + v_{k+1}}{2}$.
Throughout this section, we will study the generating functions
$F_{(\boldsymbol{\sigma},\id)}$ of $\Phi^m$
with a linear change of coordinates:
let $F^m(w) := F_{(\boldsymbol{\sigma},\id)}\circ A_m^{-1}(w)$.
Given a tuple $\boldsymbol{\delta}:=(\delta_1,\dotsc,\delta_m)$,
$x\in\C^d$ and a radius $r>0$, we denote by
$B_r(x,\boldsymbol{\delta})\subset (\C^d)^m$ the open set
\begin{multline*}
    B_r\left(\frac{x +\delta_1(x)}{2}\right) \times
    B_r\left(\frac{\delta_1 +\delta_2\circ\delta_1(x)}{2}\right)
    \times\dotsb \\
    \dotsb\times
    B_r\left(\frac{\delta_{m-2}\circ\cdots\circ\delta_1(x) +
    \delta_{m-1}\circ\cdots\circ\delta_1 (x)}{2}\right),
\end{multline*}
that is $B_r(x,\boldsymbol{\delta}) = \prod_j B_r(w_j)$ where the
$m$-tuple $\mathbf{w}$ is associated to the discrete trajectory
$(x,\delta_1(x),\dotsc,\delta_{m-1}\circ\cdots\circ\delta_1(x))$
of the discrete dynamic of $\boldsymbol{\delta}$.

\begin{lem}\label{lem:ballcrossing}
    Let $\boldsymbol{\sigma}:=(\sigma_1,\dotsc,\sigma_n)$ be such an $n$-tuple
    and $x\in\C^d$ be a fixed point of $\sigma_n\circ\dotsb\circ\sigma_1$.
    Suppose there exists a sequence $(m_j)_{j\geq 0}$
    such that there exists a sequence $(\mathbf{w}^j)_{j\geq 0}$
    with $\mathbf{w}^j\in B_r(x,(\boldsymbol{\sigma}^{m_j},\id))\setminus 
    B_{r/2}(x,(\boldsymbol{\sigma}^{m_j},\id))$
     satisfying,
     \begin{equation}\label{eq:lim}
         \left|\nabla F^{m_j}\left(\mathbf{w}^j\right)\right|^2 
        = \sum_{k=1}^{m_j n +1}
        \left|\partial_{w_k}F^{m_j}\left(\mathbf{w}^j\right)\right|^2
        \xrightarrow{j\to\infty} 0.
    \end{equation}
    Let $(a_1,\dotsc,a_n)\in(\C^d)^n$ be such that
    $B_r(x,\boldsymbol{\sigma}) = B_r(a_1)\times\cdots\times B_r(a_n)$.
    Then, there exists a sequence $(z_j)_{j\in\Z}\in(\C^d)^\Z$ and
    some integer $1\leq q\leq n$ such that
    $z_{j+1} = \sigma_j(z_j)$ 
    with 
    \begin{equation}\label{eq:in}
        \left\{
            \begin{array}{l l}
            \frac{z_j + \sigma_j (z_j)}{2} \in
            \overline{B^{2d}_r(a_{j\bmod n})}
            & \text{ for all } j\in\Z, \\
            \frac{z_q + \sigma_q (z_q)}{2}\not\in B^{2d}_{r/2}(a_q)
            \text{ or } z_q =z_1 \not\in
            B^{2d}_{r/2}(x).
            &
            \end{array}
        \right.
    \end{equation}
\end{lem}

Remark that Proposition~\ref{prop:gf} and (\ref{eq:lim}) imply
\begin{equation*}
    |z^j_k - \sigma_{k-1}(z^j_{k-1})| \xrightarrow{j\to\infty} 0 \text{ for }
    1 < k \leq m_j n +1
    \quad \text{and}\quad
    |z^j_1 - z^j_{m_j n +1}| \xrightarrow{j\to\infty} 0,
\end{equation*}
where $\mathbf{z}^j$ is the discrete trajectory associated to $\mathbf{w}^j$
\emph{via} relations (\ref{eq:egf}).
Indeed, $\partial_{v_k}F^{m_j} = \frac{1}{2}(\partial_{w_k}F^{m_j}
+\partial_{w_{k-1}}F^{m_j})$.
Thus, the proof essentially consists in an elementary application
of the Cantor's diagonal argument to ultimately get
a discrete trajectory of the dynamic $\sigma_n\circ\cdots\circ\sigma_1$
whose special property (\ref{eq:in}) comes from the 
domain of the $\mathbf{w}^j$'s.

\begin{proof}
    We first prove the case where $(m_j)_{j\geq 0}$ admits a bounded
    subsequence for a better understanding of the general case.
    Taking an extracted subsequence, we might suppose
    that $m_j \equiv m\in\N^*$.
    Then by relative compactness, we might suppose that
    $\mathbf{w}^j \to \mathbf{w} \in
    \overline{B_r(x,(\boldsymbol{\sigma}^m,0))}\setminus 
    B_{r/2}(x,(\boldsymbol{\sigma}^m,0))$.
    Let us define $(z'_j)_{1\leq j\leq mn+1}\in(\C^d)^{mn+1}$ by the relations
    (\ref{eq:egf}) for $\mathbf{w}$.
    According to Proposition~\ref{prop:gf},
    since $\nabla F^m(A_0^{-1}\mathbf{v})=0$, one has
    $z'_{j+1}=\sigma_{j \bmod n}(z'_j)$
    for $1\leq j\leq mn$ and
    $z'_1 = z'_{mn+1}$.
    Since $w\not\in B_{r/2}(x,(\boldsymbol{\sigma}^m,0))$,
    there is some integer $1\leq q'\leq mn+1$
    such that 
    $w_{q'}\not\in B^{2d}_{r/2}(a_{q'\bmod n})$ if $q'\neq mn+1$
    or $w_{q'}\not\in B^{2d}_{r/2}(x)$ otherwise.
    If $q'\neq mn+1$,
    let $k\in\N$ be such that $kn+1\leq q' <(k+1)n+1$
    and let $1\leq q\leq n$ be the integer $q = q'-kn$.
    The wanted sequence $(z_j)_{j\in\Z}$ is then the $mn$-periodic sequence
    such that 
    \begin{equation*}
        z_j = z'_{j+kn}, \quad -kn +1\leq j\leq (m-k)n.
    \end{equation*}
    In this case, 
    \begin{equation*}
        \frac{z_q + \sigma_q (z_q)}{2} = w_{q'} \not\in B^{2d}_{r/2}(a_q).
    \end{equation*}
    If $q' = mn+1$, then the wanted sequence is the $mn$-periodic sequence
    such that $z_j = z'_j$ for $1\leq j\leq mn$ with $q=1$ and in this case 
    \begin{equation*}
        z_q = z_1 = w_{mn+1} \not\in B^{2d}_{r/2}(x).
    \end{equation*}

    Now suppose that $(m_j)_{j\geq 0}$ admits no bounded infinite subsequence.
    Taking an extracted subsequence, we might suppose that $(m_j)$ is
    increasing.
    For all $j\in\N$, let $1\leq q'_j< m_j n +1$ be such that
    $w^j_{q'_j}\not\in B_{r/2}(a_{q'_j\bmod n})$
    or $q'_j = m_j n +1$ if such an integer does not exist.
    Similarly to the bounded case,
    we first suppose that we can take an extracted subsequence
    $q'_j \neq m_j n +1$ for all $j\geq 0$.
    Let $k_j\in\N$ be such that $k_j n +1 \leq q'_j < (k_j+1)n+1$
    and let $1\leq q_j \leq n$ be the integer $q_j = q'_j - k_j n$.
    Taking an extracting subsequence, we might suppose $q_j \equiv q$.
    For all $j\geq 0$, let $(\mathbf{w}'^j)\in (\C^d)^\Z$
    be the $m_j n$-periodic sequence such that
    $w'^j_{k} = w^j_{k+k_j n}$ for $-k_j n+1\leq j\leq (m_j-k_j)n$.

    Let $M\in\N^*$ and let us consider the sequence
    $(\mathbf{w}^{M,j})_{j\geq 0}$ in $(\C^{d})^{2M+1}$
    defined by restriction: for all $j\geq 0$,
    $(w^{M,j}_k)_{-M-1\leq k\leq M}:=(w'^j_k)_{-M-1\leq k\leq M}$.
    Now, we can extract a subsequence $j^M_p \to \infty$,
    such that $(\mathbf{w}^{M,j^M_p})_p$ converges.
    Since $\partial_{w_{k+k_j n}}F^{m_j}(\mathbf{w}^j)\to 0$ for
    all $-M-1\leq k\leq M$,
    the associated $(2M+1)$-tuple $(z_j)_{-M\leq j\leq M}$,
    now satisfies $z_{j+1} =\sigma_j(z_j)$
    for $-M\leq j\leq M-1$.
    By a diagonal extraction associated to
    subsequences $(j^M_p)_p$ as $M$ goes to infinity,
    we extend our $(2M+1)$-tuples $(z_j)$
    to a sequence in $\Z$ with the wanted properties.
    In particular $\frac{z_q +\sigma_q (z_q)}{2} \not\in B^{2d}_{r/2}(a_q)$.

    If one cannot extract a subsequence such that $q'_j\neq m_j n+1$,
    we can extract a subsequence such that $q'_j \equiv m_j n +1$.
    Then take $q=1$ and define $(\mathbf{w}'^j)$ to be the $m_j n$-periodic
    sequence such that $w'^j_k = w^j_k$ for $1\leq k\leq m_j n$.
    By the same way as above, one gets the wanted $(z_j)$
    by a diagonal extraction and in this case
    $z_q = z_1 \not\in B^{2d}_{r/2}(x)$.
\end{proof}

\begin{thm}\label{thm:ballcrossing}
    Let $\Phi\in\ham(\C^d)$ admitting $C^1$-small $n$-tuples $\boldsymbol{\sigma}$.
    Suppose that $x\in\C^d$ is a fixed point  of $\Phi$
    which is isolated as an invariant set.
    Then for sufficiently small $r>0$, there exists $c_\infty >0$
    and a $n$-tuple $\boldsymbol{\sigma}$ associated to $\Phi$,
    with $n$ even,
    such that for all $m\geq 1$, any gradient or reverse-gradient flow line
    $\mathbf{u}:\R\to (\C^d)^{mn+1}$,
    $\dot{\mathbf{u}}=\pm\nabla F^m(\mathbf{u})$, with
    $\mathbf{u}(0)\in \partial B_r(x,(\boldsymbol{\sigma}^m,\id))$
    and $\mathbf{u}(\tau)\in B_{r/2}(x,(\boldsymbol{\sigma}^m,\id))$
    for some $\tau\in\R$ satisfies
    \begin{equation*}
        |F^m(\mathbf{u}(0)) - F^m(\mathbf{u}(\tau))| > c_\infty.
    \end{equation*}
\end{thm}

\begin{proof}
    Since $x$ is isolated as an invariant set, there exists some $R>0$ such that
    for all $z\in B^{2d}_R(x)\setminus\{x\}$, there exists $k\in\Z$ such that
    $\Phi^k(z)\not\in B^{2d}_R(x)$.
    Fix such an $R>0$ and choose an even
    tuple $\boldsymbol{\sigma}=(\sigma_1,\dotsc,\sigma_n)$
    such that $|z-\sigma_j(z)|<R/8$ for all $z\in B^{2d}_R(x)$.
    Let $m\geq 1$ and $\mathbf{u}:\R_+ \to (\C^d)^{mn+1}$ be as the statement of the theorem,
    we may suppose that $\mathbf{u}$ takes its values in
    $B_{R/2}(x,(\boldsymbol{\sigma}^m,\id))$.
    Let $\tau > 0$ be such that
    $\mathbf{u}(\tau)\in\partial B_{R/4}(x,(\boldsymbol{\sigma}^m,\id))$.
    In order to prove the theorem, it is enough to show that
    there exists $c_\infty>0$ independent of $m\geq 1$ and $\mathbf{u}$ satisfying
    \begin{equation*}
        |F^m(\mathbf{u}(0)) - F^m(\mathbf{u}(\tau))| > c_\infty.
    \end{equation*}

    By contradiction, suppose there exists a sequence $(m_j)_{j\geq 0}$ and
    a sequence of gradient or reverse-gradient flow lines
    $\mathbf{u}^j :[0,\tau_j]\to B_{R/2}(x,(\boldsymbol{\sigma}^m,\id))$,
    $\dot{\mathbf{u}}^j = \pm\nabla F^{m_j}(\mathbf{u}^j)$,
    with $\mathbf{u}^j(0)\in\partial B_{R/2}(x,(\boldsymbol{\sigma}^m,\id))$
    and $\mathbf{u}^j(\tau_j)\in\partial B_{R/4}(x,(\boldsymbol{\sigma}^m,\id))$
    such that
    \begin{equation*}
        \left| F^{m_j}(\mathbf{u}^j(0)) - F^{m_j}(\mathbf{u}^j(\tau_j)) \right| \to 0.
    \end{equation*}
    For some $1\leq k_j\leq m_j n+1$,
    one has $|u^{j}_{k_j}(0) - u^j_{k_j}(\tau_j)|\geq R/4$ so
    \begin{equation*}
        R/4 \leq \int_0^{\tau_j} |\dot{u}^j_{k_j}(s)|\ud s
        \leq \int_0^{\tau_j} |\dot{\mathbf{u}}^j(s)|\ud s
        = \int_0^{\tau_j} |\nabla F^{m_j}(\mathbf{u}^j(s))|\ud s,
    \end{equation*}
    but
    \begin{multline*}
        \left(\int_0^{\tau_j}|\nabla F^{m_j}(\mathbf{u}^j(s))|\ud s\right)^2
        \leq \tau_j\int_0^{\tau_j}|\nabla F^{m_j}(\mathbf{u}^j(s))|^2\ud s \\
        = \tau_j \left(F^{m_j}(\mathbf{u}^j(0)) - F^{m_j}(\mathbf{u}^j(\tau_j))\right),
    \end{multline*}
    thus $\tau_j \to +\infty$. Combined with
    $\int_0^{\tau_j}|\nabla F^{m_j}(\mathbf{u}(s))|^2\ud s \to 0$,
    it implies that there exists a sequence $(s_j)_{j\geq 0}$
    with $s_j \in [0,\tau_j]$, such that
    the sequence $(\mathbf{u}^j(s_j)))_{j\geq 0}$ satisfies
    the hypothesis of Lemma~\ref{lem:ballcrossing} with $r=R/2$.

    Therefore, according to Lemma~\ref{lem:ballcrossing}, there exists a sequence
    $(z_j)_{j\in\Z}\in(\C^d)^\Z$ and some integer $1\leq q\leq n$,
    such that $|z_j+\sigma_j(z_j)-2x|\leq R$
    which implies
    that $|z_j - x|\leq R/2 + R/16 < R$
    by the specific choice of $\sigma_j$,
    $|z_q - x| > R/4$
    and $z_{j+1} = \sigma_j (z_j)$.
    Thus, for all $k\in\Z$, $\Phi^k(z_1) = z_{kn+1} \in B^{2d}_R (x)$ with 
    $z_1 \neq x$
    since
    \begin{equation*}
        \sigma_{q-1}\circ\cdots\circ\sigma_1 (z_1) = z_q
        \neq x = \sigma_{q-1}\circ\cdots\circ\sigma_1 (x) ,
    \end{equation*}
    a contradiction.
\end{proof}

\subsection{Crossing energy theorem in $\CP^d$}

We employ the notation of Section~\ref{se:proof}.  We recall that
$\boldsymbol{\sigma}=(\sigma_1,\dotsc,\sigma_{n_1})$ is a specific $n_1$-tuple,
with $n_1$ even, associated to $\Phi$, $\boldsymbol{\delta}_t =
(g_{t/(n_2-1)},\dotsc,g_{t/(n_2-1)},\id)$ is a $n_2$-tuple, with $n_2$ odd, associated to
$e^{-2i\pi t}$, $F^m_{1,t} = F_{(\boldsymbol{\sigma}^m,\boldsymbol{\delta}_t)}$
is a conical generating function of the conical Hamiltonian diffeomorphism
$e^{-2i\pi t}\Phi^m$, $M_m := \{(t,[z])\in I\times\CP^{N(m)}\ |\ F^m_t(z) =
0\}$ is the domain of the projection map $\action^m:M_m\to I$ with $N(m) =
(d+1)(n_1 m+n_2)-1$.  Similarly to the $\C^d$-case, we apply a linear change of
coordinates and study the function $F^m_t := F^m_{1,t}\circ A_m^{-1}$ and by a
slight abuse of notation we will still denote by $M_m$ and $\action^m$ domains
and functions seen in the induced projective chart.

The proof of the crossing energy theorem in $\CP^d$
follows the same lines as the $\C^d$ case.
First, we need an analogue to Lemma~\ref{lem:ballcrossing}.
We have to define a neighborhood of $\CP^{N(m)}$
similar to $B(x,(\boldsymbol{\sigma}^m,\id))$ in the $\C^d$ case.
Let $B_1\subset \C^{d+1}$ be the unit euclidean ball
centered at the origin,
so that, for $k\in\N^*$, $\partial (B_1^k)\subset(\C^{d+1})^k$ denotes
the sphere
\begin{equation*}
    \partial (B_1^k) = \bigcup_{1\leq j\leq k} B_1^{j-1}\times \sphere{2d+1}
    \times B_1^{k-j}.
\end{equation*}
Let $\pi_m : \partial(B_1^{mn_1 +n_2})\to\CP^{N(m)}$
be the quotient map by the diagonal action of $S^1$.
We define now a $S^1$-equivariant neighborhood
in the sphere $\partial(B_1^{mn_1+n_2})$
of the normalized $w$-coordinates of some point $x\in\C^{d+1}\setminus 0$
relative to $F^m_t$.
Let $\mathbf{a}=(a_1,\dotsc,a_{mn_1+n_2})\in(\C^{d+1})^{mn_1+n_2}$ be
the $w$-coordinates of $x$, that is
$B_r(x,(\boldsymbol{\sigma}^m,\boldsymbol{\delta}_t)) = \prod_j B_r(a_j)$.
Let $\lambda>0$ such that $\lambda \mathbf{a}\in\partial(B_1^{mn_1+n_2})$.
For $r>1$ we define
\begin{equation*}
    U_r(x,m,t) := S^1\cdot B_r(\lambda x,(\boldsymbol{\sigma}^m,\boldsymbol{\delta}_t))\cap
    \partial(B_1^{mn_1+n_2}),
\end{equation*}
where $S^1\cdot E := \{ \mu z\ |\ z\in E,\ \mu\in S^1\}$ for any
subset $E\subset (\C^d)^{mn_1 +n_2}$.
Let $V_r(x,m,t)\subset \CP^{N(m)}$ be the projection of
this neighborhood on $\CP^{N(m)}$.

\begin{lem}\label{lem:projballcrossing}
    Let $x\in\C^{d+1}\setminus 0$ be a fixed point of $\Phi$.
    Suppose there exists an increasing sequence
    of positive integers $(m_j)_{j\geq 0}$
    such that there exist
    a sequence $(t_j)_{j\geq 0}$ in $I$
    satisfying $(t_j)\to t\in\{0,1\}$ and
    a sequence $(\mathbf{w}^j)_{j\geq 0}$
    with $\mathbf{w}^j\in U_r(x,m_j,t)\setminus 
    U_{r/2}(x,m_j,t)$
     satisfying,
     \begin{equation*}
         \left|\nabla F^{m_j}_{t_j}\left(\mathbf{w}^j\right)\right|^2 
        = \sum_{k=1}^{m_j n_1 +n_2}
        \left|\partial_{w_k}F^{m_j}_{t_j}\left(\mathbf{w}^j\right)\right|^2
        \xrightarrow{j\to\infty} 0.
    \end{equation*}
    Let $\mathbf{a}:=(a_1,\dotsc,a_{n_1})\in(\C^{d+1})^{n_1}$ be such that
    $B_r(x,\boldsymbol{\sigma}) = \prod_j B_r(a_j)$
    and $\tilde{\mathbf{a}}:=(\tilde{a}_1,\dotsc,\tilde{a}_{n_2})
    \in(\C^{d+1})^{n_2}$ be the $n_2$-tuple
    \begin{equation*}
        \tilde{a}_k := \frac{ g_{(k-1)t/(n_2-1)}(x) +
        g_{kt/(n_2-1)}(x)}{2} \ \text{ for } \
        1\leq k\leq n_2-1 \ \text{ and } \ \tilde{a}_{n_2} = x.
    \end{equation*}
    Then, there exists a possibly infinite integer $\kappa\in\Z\cup\{+\infty\}$
    such that
    there exists a sequence $(b_j)\in(\C^{d+1})^\Z$ defined by
    \begin{equation*}
        b_j :=
        \begin{cases}
             a_{j\bmod n_1}&
            \text{if } j\leq \kappa n_1,\\
             \tilde{a}_{j-\kappa n_1}&
            \text{if } \kappa n_1 +1\leq j\leq \kappa n_1 + n_2,\\
             a_{j-n_2\bmod n_1}&
            \text{if } j\geq \kappa n_1 + n_2 +1,
        \end{cases}
    \end{equation*}
    a sequence $(z_j)_{j\in\Z}\in(\C^{d+1}\setminus 0)^\Z$ satisfying
    \begin{equation*}
        z_{j+1} =
        \begin{cases}
             \sigma_{j\bmod n_1}(z_j)&
            \text{if } j\leq \kappa n_1,\\
             g_{t/n_2}(z_j)&
            \text{if } \kappa n_1 +1\leq j\leq \kappa n_1 + n_2,\\
             \sigma_{j-n_2\bmod n_1} (z_j)&
            \text{if } j\geq \kappa n_1 + n_2 +1,
        \end{cases}
    \end{equation*}
    and some integer $1\leq q\leq n_1+n_2$ such that
    \begin{equation*}
        \left\{
            \begin{array}{l l}
                \frac{z_j +z_{j+1} }{2} \in
            \C\cdot\overline{B^{2d}_r(b_j)}
            & \text{ for all } j\in\Z, \\
            \frac{z_q + z_{q+1}}{2}\not\in \C\cdot B^{2d}_{r/2}(b_q)
            &
            \end{array}
        \right.
    \end{equation*}
\end{lem}

\begin{proof}
The proof goes on the same lines as Lemma~\ref{lem:ballcrossing}
with just additional calligraphic difficulties,
we will only underline the key changes.

Let $x\in\C^{d+1}\setminus 0$, $\mathbf{a}\in(\C^{d+1})^{n_1}$ and
$\tilde{\mathbf{a}}\in(\C^{d+1})^{n_2}$ satisfying the assumptions of the lemma.
Let $\lambda>0$ be such that $(\lambda
\mathbf{a},\lambda\tilde{\mathbf{a}})\in\partial(B_1^{n_1+n_2})$
(it exists since $x\neq 0$),
then
\begin{equation*}
    U_r(x,m,t) = S^1\cdot \left[\left(\prod_{k=1}^{n_1} B^{2(d+1)}_r(\lambda a_k)\right)^m
    \times \prod_{k=1}^{n_2} B^{2(d+1)}_r(\lambda \tilde{a}_k) \right]\cap 
    \partial (B_1^{mn_1+n_2}).
\end{equation*}
Let $(\mathbf{w}^j)$ be satisfying the assumptions of the lemma.
By $S^1$-invariance of the function $|\partial F^{m_j}_{t_j}|$
and the neighborhood $U_r(x,m_j,t)$,
we can suppose that
\begin{equation*}
    \mathbf{w}^j \in \left(\prod_{k=1}^{n_1} B^{2(d+1)}_r(\lambda a_k)\right)^{m_j}
    \times \prod_{k=1}^{n_2} B^{2(d+1)}_r(\lambda \tilde{a}_k).
\end{equation*}
The result follows from Cantor's diagonal argument applied to
the sequence $(\mathbf{w}^j/\lambda)$
in the same way as in the proof of Lemma~\ref{lem:ballcrossing}.
\end{proof}


In order to state the Crossing energy theorem in $\CP^d$,
we will need to define a ``good'' pseudo-gradient $X_m$
for the function $\action^m$.
For technical reasons,
the projection $\pi_m : \partial(B_1^{m n_1+n_2})\to\CP^{N(m)}$
is the most natural for our problem.
However the sphere $\partial(B_1^{m n_1 +n_2})$ is not smooth,
we thus introduce a smooth $S^1$-invariant sphere
$\Sigma_m \subset (\C^{d+1})^{mn_1 +n_2}$:
\begin{equation*}
    \Sigma_m := \left\{
        \mathbf{z}\in (\C^{d+1})^{mn_1+n_2} \ |\
        \sum_{k=1}^{mn_1 +n_2} |z_k|^{p_m} = 1
    \right\},
\end{equation*}
where $p_m\geq 2$ is chosen such that,
\begin{equation*}
    \forall \mathbf{z}\in\Sigma_m,
    \exists \lambda \in [1,2],\quad
    \lambda \mathbf{z} \in \partial(B_1^{mn_1 + n_2}),
\end{equation*}
(necessarily $(p_m)\to \infty$).
We endow $\CP^{N(m)}$ with the Riemannian metric
induced by the $S^1$-invariant projection
$\pi'_m :\Sigma_m\to\CP^{N(m)}$.
Since 
\begin{equation*}
\textup{dist}(\partial U_r(x,m,t),U_{r/2}(x,m,t))\geq r/2,
\end{equation*}
the condition on $p_m$ implies that
\begin{equation}\label{eq:distV}
\textup{dist}(\partial V_r(x,m,t),V_{r/2}(x,m,t))\geq r/4.
\end{equation}
Let $f^m:I\times \CP^{N(m)}\to\R$ be the $C^1$ function satisfying
$f^m(t,\pi'_m(\mathbf{z})) = F^m_t(\mathbf{z})$ for all $\mathbf{z}\in\Sigma_m$,
so that $M_m = \{ (t,\zeta)\in I\times \CP^{N(m)} \ |\
f^m(t,\zeta)= 0\}$.
The pseudo-gradient $X_m$ of $\action^m$ is defined by
\begin{equation*}\label{eq:X}
    X_m(t,\zeta) := 
    \partial_t f^m(t,\zeta) \nabla f^m(t,\zeta) -
        |\nabla f^m(t,\zeta)|^2 \frac{\partial}{\partial t}
\end{equation*}
We have $\la X_m, -\frac{\partial}{\partial t}\ra \geq 0$ 
with equality if and only if $\nabla f^m = 0$,
that is to say $\ud \action^m = 0$.

\begin{thm}\label{thm:projballcrossing}
    Let $\Phi\in\ham_\C(\C^{d+1})$ be a lift of
    $\varphi\in\ham(\CP^d)$.
    Suppose that $x\in\C^{d+1}\setminus 0$ is a fixed
    point of $\Phi$ such that $[x]\in\CP^d$ 
    is isolated as an invariant set of $\varphi$.
    Then for sufficiently small $r>0$, there exists $c_\infty >0$
    and a tuple $\boldsymbol{\sigma}$ associated to $\Phi$
    such that for all $m\geq 1$,
    if $(t,\zeta^m_t)\in M_m$ denotes the critical
    point of $\action^m$ with
    critical value $t\in\{ 0, 1\}$
    associated to $x$,
    any gradient flow line
    $\mathbf{u}:\R\to M_m$,
    $\dot{\mathbf{u}} =\pm X_m(\mathbf{u})$,
    with $\mathbf{u}(0)\in I\times \partial V_r(x,m,t)$
    and $\mathbf{u}(\tau)\in I\times V_{r/2}(x,m,t)$ for some $\tau\in\R$
    satisfies
    \begin{equation*}
        |\action^m(\mathbf{u}(0)) -\action^m(\mathbf{u}(\tau))| > c_\infty.
    \end{equation*}
    The pseudo-gradient $X_m$ can be replaced by a pseudo-gradient
    $C^0$-close to it,
    \emph{e.g.} a Morse-Smale adapted pseudo-gradient
    if $\action^m$ is a Morse function.
\end{thm}

\begin{proof}
    We follow the steps of the proof of Theorem~\ref{thm:ballcrossing}.
    By contradiction, suppose there exists a sequence $(m_j)_{j\geq 0}$
    and a sequence of pseudo-gradient flow line $\mathbf{u}^j : [0,\tau_j]\to
    U_r(x,m_j,t)$, $\dot{\mathbf{u}}^j =\pm X_{m_j}(\mathbf{u}^j)$ with
    $\mathbf{u}^j(0)\in \partial V_r (x,m_j,t)$ and
    $\mathbf{u}^j(\tau_j) \in V_{r/2} (x,m_j,t)$ such that
    \begin{equation*}
        \left|\action^{m_j}(\mathbf{u}^j(0)) -
        \action^{m_j}(\mathbf{u}^j(\tau_j))\right|
        \xrightarrow{j\to+\infty} 0
        \quad \text{ and }\quad
        \action^{m_j}(\mathbf{u}^j(0))
        \xrightarrow{j\to+\infty} 0.
    \end{equation*}
    First we must show that $\tau_j\not\to 0$.
    Let $p_2 : I\times\CP^N \to \CP^N$ be the projection
    on the second factor,
    then (\ref{eq:distV}) implies that
    \begin{equation*}
        \frac{r}{4} \leq \int_0^{\tau_j} |\ud p_2 \cdot \dot{\mathbf{u}}^j |\ud s,
    \end{equation*}
    so
    \begin{equation*}
        \left(\frac{r}{4}\right)^2 \leq \tau_j
        \int_0^{\tau_j} |\ud p_2\cdot X_{m_j}(\mathbf{u}^j)|^2 \ud s
        = \tau_j
        \int_0^{\tau_j} 
        (\partial_t f^{m_j}(\mathbf{u}^j))^2
        |\nabla_\zeta f^{m_j}(\mathbf{u}^j) |^2 \ud s.
    \end{equation*}
    Remark that there exists some $C>0$ independent of $m$
    (it only depends on $(\boldsymbol{\delta}_t)$) such that $0\leq -\partial_t f^{m} < C$,
    thus
    \begin{equation*}
        \int_0^{\tau_j} |\ud p_2\cdot X_{m_j}(\mathbf{u}^j)|^2 \ud s
        \leq C^2 \int_0^{\tau_j} |\nabla_\zeta f^{m_j}(\mathbf{u}^j)|^2\ud s.
    \end{equation*}
    This last term goes to $0$ since
    \begin{equation}
        \left|
        \action^{m_j}(\mathbf{u}^j(0))-\action^{m_j}(\mathbf{u}^j(\tau_j))
        \right| = \int_0^{\tau_j} \la -\frac{\partial}{\partial t} ,
        X_{m_j}(\mathbf{u})\ra \ud s = \int_0^{\tau_j} |\nabla_\zeta
        f^{m_j}(\mathbf{u}^j)|^2\ud s.
    \end{equation}
   Therefore, $\tau_j\to +\infty$ and thus there exists
   a sequence $(s_j)_{j\geq 0}$ in
   $I\times V_r (x,m,t)\setminus I\times V_{r/2} (x,m,t)$ such that
   $|\nabla_\zeta f^{m_j}(s_j)|\to 0$.

   Let $(t^j;\lambda_j \mathbf{w}^j)\in I\times U_r(x,m_j,t)$
    be lifted from $s_j$ with 
    $\mathbf{w}^j\in\Sigma_{m_j}$ and $\lambda_j\in[1,2]$ such that
    $\lambda_j \mathbf{w}^j\in\partial (B_1^{m_j n_1+n_2})$
    (which exists by definition of $\Sigma_{m_j}$).
    Since $t^j = \action^{m_j}(s_j)$, one has $t^j\to t$.
    Since $|\nabla_\zeta f^{m_j}(s_j)|\to 0$, the norm of the orthogonal projection of
    $\nabla F_{t^j}(\mathbf{w}^j)\in\C^{N(m_j)+1}$ on the
    sphere $\Sigma_{m_j}$ goes to zero as $j\to\infty$.
    The radial component is
    $\la \mathbf{w}^j , \nabla F_{t^j}(\mathbf{w}^j)\ra = 2
    F_{t^j}(\mathbf{w}^j) = 0$, hence
    $|\nabla F^{m_j}_{t^j} (\mathbf{w}^j) |\to 0$.
    Since $\lambda_j\in[1/2,1]$, the homogeneity of $F^{m_j}_{t^j}$
    implies that
    \begin{equation*}
        \left|\nabla F^{m_j}_{t^j} (\lambda_j \mathbf{w}^j)
        \right|\xrightarrow{j\to\infty} 0.
    \end{equation*}
    We can thus apply Lemma~\ref{lem:projballcrossing}
    to the sequences $(m_j)$, $(\lambda_j \mathbf{w}^j)$
    and the fixed point $x\in\C^{d+1}$.
    We then find a sequence $(z_j)_{j\in\Z}$ in $\C^{d+1}$
    such that $\varphi^k([z_0])$ keeps close to $[x]$
    for all $k\in\Z$ with $[z_0]\neq [x]$.
\end{proof}

\bibliographystyle{amsplain}
\bibliography{biblio} 

\end{document}